\documentclass[12pt]{amsart}
\usepackage{amsmath, amssymb}
\newfont {\cyr} {wncyr10}
\renewcommand{\labelenumi}{{(\roman{enumi})}}

\usepackage{amsfonts}
\usepackage{amsmath}
\usepackage{amssymb}
\usepackage{setspace}
\usepackage{multicol}
\usepackage{color}

\mathchardef\tnode="020E

\def\arc{ 
 \hbox{\kern -0.15em \vbox{\hrule width 2.5em height 0.6ex depth -0.5 ex} \kern
-0.33em}}

\def\darc{
 \rlap{\lower0.2ex\arc}{\raise0.2ex\arc}}

\def\tarc{
 \rlap{\rlap{\lower0.4ex\arc}{\raise0.4ex\arc}}{\arc}}

\def\stroke#1{
 \kern 0.05
\rlap\arc{{\textstyle{#1}}\atop\phantom\arc} \kern -0.22em}

\def\dstroke#1{
 \kern 0.05em
\rlap\darc{{\textstyle{#1}}\atop\phantom\darc} \kern -0.22em}

\def\centerscript#1{
 \setbox0=\hbox{$\tnode$} \hbox to
\wd0{\hss$\scriptstyle{#1}$\hss}}

\def\node{
 \def\super{} \def\sub{}
\futurelet\next\dolabellednode}

  \let\sp=^ \let\sb=_

  \def\dolabellednode{%
   \ifx\next\sb\let\next\getsub \else \ifx\next\sp\let\next\getsuper
\else\let\next\donode \fi \fi \next}

  \def\getsub_#1{\def\sub{#1}\futurelet\next\dolabellednode}

\def\getsuper^#1{\def\super{#1}\futurelet\next \dolabellednode}

  \def\donode{%
   \rlap{$\mathop{\phantom\tnode}\limits_{\centerscript{\sub}}
^{\centerscript{\super}}$}\tnode}

\def\varcdn{
 \kern
-0.03em\vbox{\kern -0.5ex \hbox to \wd0{\hss\vrule width 0.04em depth 5.8ex\hss} \kern -0.3ex \hbox{$\tnode$}}}

\def\a3{\node_1\arc\node_2\arc\node_3}

\def\c3{\node_1\arc\node_2\darc\node_3}

\def\m24{\node\arc\node\dstroke{\sim}\node} \def\u43{\node\darc\node\dstroke{\sim}\node}

\newcommand{\varcdnl}[1]{ 
\kern -0.03em\vbox{\kern -0.5ex \hbox to \wd0{\hss\vrule width 0.04em depth 5.8ex\hss} \kern -0.3ex
\hbox{$\tnode^{#1}$}}}

\def\nodef{
\def\super{} \def\sub{} \futurelet\next\dolabellednodef}

  \let\sp=^ \let\sb=_

  \def\dolabellednodef{%
  \ifx\next\sb\let\next\getsubf \else
\ifx\next\sp\let\next\getsuperf \else\let\next\donodef \fi \fi \next}

\def\getsubf_#1{\def\sub{#1}\futurelet\next\dolabellednodef}

\def\getsuperf^#1{\def\super{#1}\futurelet\next \dolabellednodef}

  \def\donodef{%
  \rlap{$\mathop{\phantom\tnodef}\limits_{\centerscript{\sub}}
^{\centerscript{\super}}$}\tnodef}

\def\varcdnf{
 \kern -0.03em\vbox{\kern -0.5ex \hbox to \wd0{\hss\vrule width 0.04em depth
5.8ex\hss} \kern -0.3ex \hbox{$\tnodef$}}}

\newtheorem{theorem}{Theorem}[section]
\newtheorem{lemma}[theorem]{Lemma}
\newtheorem{proposition}[theorem]{Proposition}

\newtheorem{definition}[theorem]{Definition}

\newcounter{claim}[theorem]

\newcounter{cclaim}[theorem]

\def \udot {{}^{\textstyle .}}

\newcommand{\E}{\mathrm{E}}\newcommand{\SU}{\mathrm{SU}}
\newcommand{\F}{\mathrm{F}}\newcommand{\B}{\mathrm{B}}
\newcommand{\M}{\mathcal{M}}

\newcommand{\C}{\mathcal{C}}
\newcommand{\Q}{\mathrm{Q}}

\newcommand{\Aut}{\mathrm{Aut}}

\newcommand{\Out}{\mathrm{Out}}

\newcommand{\Syl}{\mathrm{Syl}}\newcommand{\syl}{\mathrm{Syl}}

\newcommand{\GF}{\mathrm{GF}}
\newcommand{\GL}{\mathrm{GL}}
\newcommand{\Sp}{\mathrm{Sp}}
\newcommand{\SL}{\mathrm{SL}}

\newcommand{\PSL}{\mathrm{PSL}}
\newcommand{\PSU}{\mathrm{PSU}}
\newcommand{\Sym}{\mathrm{Sym}}
\newcommand{\Alt}{\mathrm{Alt}}

\newcommand{\U}{\mathrm{U}}
\newcommand{\Stab}{\mathrm{Stab}}

\def \qedc {$\hfill \blacksquare$\newline}

\def \L {\hbox {\rm L}}

\def \OO {\hbox {\rm O}}

\def \syl {\hbox {\rm Syl}}\def \Syl {\hbox {\rm Syl}}

\def \ov {\overline}

\def \wt {\widetilde}

\def \Aut{ \mathrm {Aut}}

\def \Out{\mbox {\rm Out}}

\def \Mat{\mbox {\rm Mat}}

\def \B{\mbox {\rm B}}

\def \M{\mbox {\rm M}}

\def \Co {\mbox {\rm Co}}

\def \POmega {\mbox {\rm P}\Omega}

\def \PSU {\mbox {\rm \PSU}}
\def \GSp {\mbox {\rm GSp}}

\def \diag {\mathrm {diag}}

\begin{document}
\renewcommand{\labelenumi}{(\roman{enumi})}

\title  {A characterisation of almost simple groups with socle ${}^2\E_6(2)$ or $\M(22)$ }
 \author{Chris Parker}
 \author{M. Reza Salarian}
  \author{Gernot Stroth}
\address{Chris Parker\\
School of Mathematics\\
University of Birmingham\\
Edgbaston\\
Birmingham B15 2TT\\
United Kingdom} \email{c.w.parker@bham.ac.uk}

\address{M. Reza Salarian \\Department of Mathematics\\
Tarbiat Moallem University\\
 University square- The end of Shahid beheshti Avenue\\
31979-37551 Karaj- Iran}

\email {salarian@iasbs.ac.ir}

\address{Gernot Stroth\\
Institut f\"ur Mathematik\\ Universit\"at Halle - Wittenberg\\
Theordor Lieser Str. 5\\ 06099 Halle\\ Germany}
\email{gernot.stroth@mathematik.uni-halle.de}

\date{\today}

\maketitle \pagestyle{myheadings}

\markright{{\sc }} \markleft{{\sc Chris Parker, M. Reza Salarian and Gernot Stroth}}

\begin{abstract} We show that the
sporadic simple group $\M(22)$, the exceptional group of Lie type ${}^2\E_6(2)$ and their automorphism groups are
uniquely determined by the approximate structure of the centralizer of an element of order $3$ together with some
information about the fusion of this element in the group.
\end{abstract}

\section{Introduction}

The aim of this article is to identify the  groups with minimal normal subgroup $\M(22)$, one of the sporadic
simple groups discovered by Fischer,  and the exceptional Lie type group ${}^2\E_6(2)$ from certain  information
about the centralizer of a certain element of order $3$.

The results of this paper and its companions \cite{Parker1,PS1, F4, ParkerRowley2} is to provide identification theorems
for the work in \cite{almostlie} where the following configuration relevant to the classification of groups with
a so-called large $p$-subgroup is considered.  We are given a group $G$, a prime $p$ and a large $p$-subgroup $Q$
(the definition of a large $p$-subgroup is not important for this discussion) and we find ourselves in the
following situation. Containing a Sylow $p$-subgroup $S$ of $G$ there is a group $H$ such that $F^*(H)$ is a
simple group of Lie type. In the typical situation when  one would expect that this group $H$ is in fact
the entire group $G$. However it can exceptionally happen that in fact the normalizer of the large subgroup is
not contained in $Q$. This happens more frequently than one might expect when  $F^*(H)$ is defined over the field
of $2$ or $3$ elements and $N_H(Q)$ is soluble. Indeed in \cite{almostlie}, the authors determine all the cases
when this phenomena appears. This paper fits into the picture when we consider $F^*(H) \cong \Omega_7(3)$. In
$H$, the large subgroup $Q$ is extraspecial of order $3^7$ an $N_{F^*(H)}(Q) \approx 3^{1+6}_+.(SL_2(3)\times \Omega_3(3)).$ In \cite{almostlie} we show that if $N_G(Q)$ is not contained in $H$, then we must have $C_H(Z(Q))$ is a centralizer in a group of type either $\M(22)$ or ${}^2\E_6(2)$ where these centralizers are defined as follows.

\begin{definition} We say that  $X$ is similar to a $3$-centralizer in a group of type  ${}^2\E_6(2)$
provided
\begin{enumerate}
\item $Q=F^*(X)$ is extraspecial of order $3^{1+6}$ and $Z(F^*(X)) =Z(X)$; and
\item $O_{2}(X/Q) \cong \Q_8\times \Q_8\times \Q_8$.
\end{enumerate}
\end{definition}

\begin{definition} We say that  $X$ is similar to a $3$-centralizer in a group of type  $\M(22)$
provided
\begin{enumerate}
\item $Q=F^*(X)$ is extraspecial of order $3^{1+6}$ and $Z(F^*(X)) =Z(X)$; and
\item $O_{2}(X/Q)$ acts on $Q/Z$ as a subgroup of order $2^7$ of $\Q_8\times \Q_8\times \Q_8$, which contains
$Z(\Q_8\times \Q_8\times \Q_8)$.
\end{enumerate}
\end{definition}

 In this paper we will prove the following two theorems

\begin{theorem}\label{Main}
Suppose that $G$ is a group, $H \le G$ is similar to a $3$-centralizer in a group of type ${}^2\E_6(2)$, $Z = Z(F^*(H))$ and $H= C_G(Z)$. If $S \in \Syl_3(G)$ and $Z$ is not weakly closed in $S$ with respect to $G$, then $Z$ is not weakly closed in $O_3(H)$ and $G \cong {}^2\E_6(2)$, ${}^2\E_6(2).2$, ${}^2\E_6(2).3$ or ${}^2\E_6(2).\Sym(3)$.
\end{theorem}

\begin{theorem}\label{Main1}
Suppose that $G$ is a group, $H \le G$ is similar to a $3$-centralizer in a group of type $\M(22)$, $Z =
Z(F^*(H))$ and $H= C_G(Z)$. If $S \in \Syl_3(G)$ and $Z$ is not weakly closed in $S$ with respect to $G$, then $Z$ is not weakly closed in $O_3(H)$ and $G
\cong \M(22)$ or $\Aut(\M(22))$.
\end{theorem}

A minor observation that is useful to us in our forthcoming work on $\M(23)$ and the Baby Monster $\F_2$ is that the interim statements that we prove in this paper become observations about the structure of $\M(22)$ and ${}^2\E_6(2)$ once the main theorems have been proved.

The paper is organised as follows. In Section 2, we  gather together facts about the 20-dimensional $\GF(2)\U_6(2)$-module, centralizers of involutions in this group and in the spit extension $2^{20}:\U_6(2)$ as well as a transfer theorem for groups of shape $2^{10}.\Aut(\Mat(22))$. We close Section 2 with a collection of theorems and lemmas which will be applied in the proof of our main theorems.

Section 3 contains a proof of the following theorem which we used to determine the structure of the centralizer of an involution in groups satisfying the hypothesis of Theorem~\ref{Main}.

\begin{theorem}\label{Not3embedded} Suppose that $X$ is a group, $O_{2'}(X)= 1$, $H = N_X(A)= AK$  with $H/A \cong K \cong \U_6(2)$ or $\U_6(2):2$, $|A|=2^{20}$ and $A$
a minimal normal subgroup of $H$. Then $H$ is not a strongly $3$-embedded subgroup of $X$.
\end{theorem}

In Section 3, we set $H= C_G(Z)$ and $Q= O_3(H)$ and start by investigating the possible structure of $H$. Almost immediately from the hypothesis we know that $H/O_3(H)$ embeds into $\Sp_2(3)\wr \Sym(3)$.  Lemma~\ref{ZnotweakQ} shows that $Z$ is not weakly closed in $Q$  and we use this information to build a further $3$-local subgroup $M$. It turns out that $M$ is the normalizer of the Thompson subgroup of a Sylow $3$-subgroup of $G$ contained in $H$  and further Lemma~\ref{N_G(J)} that $O_3(M)$ elementary abelian of order either $3^5$ or $3^6$ and $F^*(M/O_3(M)) \cong \Omega_5(3)$.

Section 5 is devoted to the proof of Theorem~\ref{Main1}. From the information gathered in Section 3 we quickly show that the centralizer of an involution has shape $2\udot \U_6(2)$ or $2\udot\U_6(2).2$. From this we can build a further $2$-local subgroup of shape $2^{10}:\Mat(22)$ or $2^{10}:\Aut(\Mat(22))$ and use Lemma~\ref{autm22} to show that $G$ has a subgroup of index $2$ in the latter case. Finally we apply
\cite[Theorem 31.1]{Asch} to finally prove Theorem~\ref{Main1}.

From Section 7 onwards we may assume that $H$ is a $3$-centralizer in a group of type ${}^2\E_6(2)$. In particular, we have that $O_{2}(H/Q) \cong \Q_8 \times \Q_8 \times \Q_8$ and we let $r_1$ be an involution in $H$ such that $r_1Q$ is contained in the first direct factor.
By the end of Section 7 we know $r_1$ is a $2$-central involution which  contains an extraspecial subgroup of order $E\cong 2^{1+20}_+$ in its centralizer and that $F^*(N_G(E)/E)\cong \U_6(2)$. Our next objective is to control the embedding of $N_G(E)$ in $C_G(r_1)$ so that we can show that $C_G(r_1)= N_G(E)$. To do this we first transfer elements of order $2$ and order $3$ from $G$. The transfer of an element of order 2 is carried out in Section 8 and then the element of order 3 easily follows in Section 9. At this stage we know that $ N_G(E) \approx 2^{1+20}_+.\U_6(2)$, however we still don't know enough about the centralizers of elements of order $3$ in $C_G(r_1)$ to be able to show that $N_G(E)$ is strongly 3-embedded in $C_G(r_1)$. Thus in Section 10, we determine the centralizer of a further element of order $3$ with the help of Astill's Theorem \cite{Astill}. With this we can prove that $N_G(E)$ is indeed strongly $3$-embedded in $ C_G(r_1)$ and conclude from Theorem~\ref{Not3embedded} that $C_G(r_1) = N_G(E)$. At this stage, we could apply Aschbacher's Theorem ~\cite{Aschbacher2E6} to identify $G$, however, partly because some of the background material about the simple connectivity of certain graphs related to geometries to type $\F_4$  has not yet been published and also because we would prefer a uniform building theoretic approach to the classification of the groups such as ${}^2\E_6(2)$, in the penultimate section we identify the ${}^2\E_6(2)$ by showing that the coset geometry constructed from certain $2$-local subgroups containing the normalizers of a Sylow 2-subgroup of $G$ is in fact a chamber system of type $\F_4$. The Tit's Local Approach Theorem yields that the group generated by these $2$-local subgroups is $\F_4(2)$. Finally we apply Holt's Theorem \cite{Ho} to see that $G \cong {}^2\E_6(2)$. Combining this with the transfer arguments presented earlier finally proves Theorem~\ref{Main} the details being presented in our brief final section.

Throughout this article we follow the now standard Atlas \cite{Atlas} notation for group extensions. Thus $X\udot
Y$ denotes a non-split extension of $X$ by $Y$, $X{:}Y$ is a split extension of $X$ by $Y$ and we reserve the
notation $X.Y$ to denote an extension of undesignated type (so it is either unknown, or we do not care). Our
group theoretic notation is mostly standard and follows that in  \cite{GLS2} for example.  For odd primes $p$,
the extraspecial groups of exponent $p$ and order $p^{2n+1}$ are denoted by $p^{1+2n}_+$. The extraspecial
$2$-groups of order $2^{2n+1}$ are denoted by $2^{1+2n}_+$ if the maximal elementary abelian subgroups have order
$2^{1+n}$ and otherwise we write $2^{1+2n}_-$.  The extraspecial group of order $8$ is denoted by $\Q_8$. We
expect our notation for specific groups is self-explanatory. For a subset $X$ of a group $G$, $X^G$ denotes that
set of $G$-conjugates of $X$. If $x, y \in H \le  G$, we  write $x\sim _Hy$ to indicate that $x$ and $y$ are
conjugate in $H$.  Often we shall give suggestive descriptions of groups which indicate the isomorphism type of
certain composition factors. We refer to such descriptions as the \emph{shape} of a group. Groups of the same
shape have normal series with isomorphic sections. We use the symbol $\approx$ to indicate the shape of a group.

\noindent {\bf Acknowledgement.}  The initial  work on this paper was prepared during  a visit of the first and third author
to the Mathematisches Forschungsinstitut Oberwolfach as part of the
Research in Pairs Programme, 30th November--12 December, 2009. The authors are pleased to thank the MFO and its staff for the pleasant and
stimulating research environment that they provided. The first author is also grateful to the DFG for support and the mathematics department in Halle for their hospitality.

\section{Preliminary facts}

Suppose that $X = \U_6(2){:}2$, $Y = \U_6(2)$,  $\ov X=\SU_6(2){:}2$,  $\ov Y= \SU_6(2)$ and $W$ is the natural
$\GF(4)\ov Y$-module. Let $\{w_1, \dots, w_6\}$ be a unitary basis for $W$.  Note that $\ov X$ acts on $W$ with
the outer elements acting as semilinear transformations. Let $\ov M$ be the monomial subgroup of $\ov Y$ of shape
$3^5{:}\Sym(6)$ and $M$ be its image in $Y$. Set $J= O_3(M)$. Then $J$ is elementary abelian of order $3^4$ and
$\ov J$ is elementary abelian of order $3^5$. Note that $M$ contains a Sylow $3$-subgroup of $Y$.  We  let $e_1$,
$e_2$ and $e_3$ be  the images of the diagonal matrices $\diag (\omega, \omega^{-1},1,1,1,1)$,  $\diag (\omega,
\omega,\omega^{-1}, \omega^{-1},1,1)$ and $\diag(\omega, \omega,\omega,\omega^{-1},\omega^{-1},\omega^{-1})$ in $
Y$ respectively. Then $e_1, e_2$ and $e_3$ are representatives of the three conjugacy classes of elements of
order $3$ in $Y$.

\begin{lemma}\label{U62J} Every element of order $3$ in $X$ is $X$-conjugate to an element of $J$ and the centralizers of
elements of order $3$ are as follows.
\begin{enumerate}
\item $C_Y(e_1) \cong 3 \times \SU_4(2)$;\item $C_Y(e_2) \cong 3 \times \Sym(3)\wr 3$ and has order $2^3.3^5$;
and \item  $C_Y(e_3) \cong (\SU_3(2)\circ \SU_3(2)).3 \approx 3^{1+4}_+.(\Q_8\times \Q_8).3$.\end{enumerate}
\end{lemma}
\begin{proof} Given the descriptions of $e_1$, $e_2$ and $e_3$ above this is an easy calculation. (See also \cite[(23.9)]{Asch} and correct the typographical error.)\end{proof}

We also need to know the centralizers of involutions in $X$.

\begin{lemma}\label{centralizerinvsU6}
 $X$ has five conjugacy classes of involutions and their centralizers have shapes as follows.
\begin{eqnarray*}
C_X(t_1) &\approx& 2^{1+8}_+:\SU_4(2). 2; \\
C_X(t_2) &\approx& 2^{4+8}.(\Sym(3) \times \Sym(3)).2; \\
C_X(t_3) &\approx& 2^{9}.3^2.\Q_8.2 \le 2^9:\L_3(4).2;\\
C_X(t_4) &\approx& 2\times \Sp_6(2); \text{ and}\\
C_X(t_5) &\approx& 2 \times (2^5:\Sp_4(2)).\\
\end{eqnarray*}
The involutions $t_1,t_2$ and $t_3$ are contained in $Y$ and their centralizers in $Y$ are obtained by dropping
the final $2$ in their description in $X$. Furthermore we may suppose that $t_5= t_4t_1$ and $C_X(t_5) \le C_X(t_4)$.
\end{lemma}

\begin{proof}  This can be found in \cite{AschSe} for the involution $t_1, t_2$ and $t_3$ (see also \cite[(23.2)]{Asch} and the following discussion). For the involutions
$t_4$ and $t_5$ we refer to \cite[Proposition 4.9.2]{GLS3}.
\end{proof}

We note that the involutions  $t_1$, $t_2$, and $t_3$ are the images in $Y$ of the involutions $\diag(t,I,I)$,
$\diag(t,t,I)$ and $\diag(t,t,t)$ respectively,  where $t = \left(\begin{array}{cc} 0&1\\1&0\end{array}\right)$
and $I$ is the $2\times 2$ identity matrix. The conjugates of $t_1$ are called \emph{unitary transvections}.

\begin{lemma}\label{nofours} There are no fours groups in $X$ all of whose non-trivial elements are unitary
transvections. In particular, if    $t$ is a unitary transvection, then $\langle t\rangle$ is weakly closed in
$O_2(C_X(t))$.
\end{lemma}

\begin{proof} Suppose that $F$ is a fours group in $X$ and that all the non-trivial elements of $F$ are unitary
transvections. Let $x_1,x_2$ and $x_3$ be the non-trivial elements of $F$. Since $C_X(x_1)$ is a maximal subgroup
of $X$ and $Z(C_X(x_1))=\langle x_1\rangle$, $X= \langle C_X(x_1),C_X(x_2)\rangle$. Therefore, $C_W(x_1) \not =
C_W(x_2)$. Let $v\in W \setminus C_W(x_1)$ and $w \in C_X(x_2)\setminus C_W(x_1)$. Then $[v,x_3] =[v,x_2]$ and
$[w,x_3] = [w,x_2]$. Hence, as $\dim [W,x_3]=1$, $[W,x_1]=[W,x_2]$ is normalized by $X$, which is a
contradiction. If $O_2(C_G(t))$  containes a unitary transvection $s$ with $t\neq s$, then conjugation in
$O_2(C_G(t))$  reveals that all elements of $\langle s,t\rangle$ are unitary transvections and this is impossible
as we have just seen. Thus $\langle t \rangle$ is weakly closed in $O_2(C_X(t))$.
\end{proof}

Let $P_1$ and $P_2$ be the connected parabolic subgroups of $Y$ containing a fixed Borel subgroup where notation
is chosen so that $$P_1\approx  2^{1+8}_+{:}\SU_4(2)$$ and $$P_2\approx  2^9{:}\L_3(4).$$

\begin{lemma}\label{unique}  Suppose that $Y \cong \U_6(2)$ and that $V$ is an irreducible $20$-dimensional
$\GF(2)Y$-module. Then $V\otimes \GF(4)$ is the exterior cube of $W$. In particular, $\dim C_V(O_2(P_2))=1$ and
$\dim C_V(e_3)=2$.
\end{lemma}
\begin{proof} First consider the restriction of $V$ to $O_3(C_Y(e_3))$. This group has no faithful characteristic $2$-representation
of dimension less than $9$ and as $e_3$ is inverted by a conjugate $t$ of $t_3$, we see that any characteristic~$2$ representation of $O_3(C_Y(e_3))\langle t\rangle$ has dimension at least $18$. It follows that $\dim
C_V(e_3) = 2$ and that $V$ is absolutely irreducible.  By Smith's Theorem \cite{Smith}, we now have, for $i=1,2$,
$C_V(O_2(P_i))$ are irreducible $P_i$-modules.  Suppose that $\dim C_V(O_2(P_2)) >1$. Then, as $P_2/O_2(P_2)\cong
\L_3(4)$ contains an elementary abelian subgroup of order 9 all of whose subgroups of order 3 are conjugate, we
have $\dim C_V(O_2(P_2)) \ge 8$. Since $t_1 \in O_2(P_2)$  and since there exists $x\in P_1$ such that $P_1=
\langle O_2(P_2),O_2(P_2)^x \rangle$, we either have $\dim C_V(t_1) \ge 15$ or $\dim C_V(P_1) \ge 2$. The latter
possibility violates Smith's Theorem. Hence $\dim C_V(t_1) \ge 15$. Thus $V/C_V(t_1)$ has dimension at most $5$.
Since $P_1/O_2(P_1)\cong \SU_4(2)$ has Sylow $3$-subgroups of order $3^4$, we have $[V,P_1] \le C_V(t_1)$ and so
$t_1$ is a transvection by Smith's Theorem. Since $t_1$ inverts $e_1$, we now have $\dim C_V(e_1) \ge 18$ and
taking a suitable product of three conjugates of $e_1$ we obtain a conjugate of $e_3$ centralizing a $14$-space
rather than a $2$-space. At which stage we conclude $\dim C_V(O_2(P_1))= 1$. Finally, using
\cite[5.5]{Aschbacher2E6} we obtain the statement of the lemma. \end{proof}

We note that the $20$-dimensional $\GF(2)Y$-module in Lemma~\ref{unique} extends to an action of $X$ (as can be
seen in the group ${}^2\E_6(2).2$). Our next gual is to determine the action of elements of $X$ on  $V$ described
in Lemma~\ref{unique}. We recall that $P_1/O_2(P_1) \cong\SU_4(2)$. We call the $4$-dimensional $\GF(4)\SU_4(2)$
viewed as an $8$-dimensional $\GF(2)$-module the \emph{unitary module} for $\SU_4(2)$ and the $6$-dimensional
$\GF(2)\SU_4(2)$-module which can be seen as the exterior square of the unitary module is called the
\emph{orthogonal module} for $\SU_4(2)$. We will also meet the \emph{symplectic module} for $C_X(t_4)/\langle
t_4\rangle \cong \Sp_6(2)$ as well as the \emph{spin module} which has dimension $8$ and this is the unique
$8$-dimensional irreducible $\Sp_6(2)$-module (see \cite[5.4]{Aschbacher2E6}). Finally, from
Lemma~\ref{centralizerinvs} we have that $C_Y(t_2)/O_2(C_Y(t_2)) \cong \mathrm \Omega_4^+(2)$ and so this group
has an orthogonal module.

\begin{proposition}\label{Vaction} Suppose that $X = \U_6(2):2$ and $V$ is the irreducible $\GF(2)X$-module of dimension $20$.
\begin{enumerate}
\item The following hold:
\begin{enumerate}
\item $\dim C_V(t_1) = 14$, $[V,t_1]$ is the orthogonal module and $C_V(t_1)/[V,t_1]$ is
the unitary module for $C_X(t_1)/O_2(C_X(t_1)) \cong \SU_4(2)$; \item  $\dim C_V(t_2) = 12$, $C_V(t_2)/[V,t_2]$ is the  orthogonal module for  $C_X(t_2)/O_2(C_X(t_2)) \cong \Omega_4^+(2)$;
\item $\dim C_V(t_4) = 14$,
 $[V,t_4]$ is the symplectic module and $C_V(t_4)/[V,t_4]$ is the spin module for  $C_X(t_4)/O_2(C_X(t_4)) \cong \Sp_6(2)$;
\item  $\dim C_V(t_3) = \dim C_V(t_5)= 10$;
\end{enumerate}
    \item The stabilizers of non-zero vectors in $V$ are as follows:
\begin{eqnarray*}
\Stab_X(v_1)&\approx&2^9:\L_3(4).2;\\
\Stab_X(v_2)&\approx&2^{1+8}.\Sp_4(2).2;\\
\Stab_X(v_3)&\approx&2^8:3^2.\Q_8.2;\\
\Stab_X(v_4)&\approx&\L_3(4).2.2; \text{ and }\\
\Stab_X(v_5)&\approx&3^{1+4}_+.(\Q_8\times \Q_8).2.2.
\end{eqnarray*}
Here $v_1,v_2,v_3$ are the singular vectors.
\end{enumerate}
\end{proposition}

\begin{proof} For the involutions $t_i$, $i=1,2,3$, $\dim [V,t_i]$ is given in \cite[7.4
(1)]{Aschbacher2E6}. In particular (i) (c) holds and the dimension statements in (i)(a) and (i)(b) hold.

The remaining parts of (i)(a) can be deduced from \cite[(5.6)]{Aschbacher2E6}.

The involution $t_2$ centralizes the image in $X$ of $\langle a,b\rangle$ where $a=\diag
(\omega,\omega,\omega^{-1},\omega^{-1},\omega,\omega^{-1})$ and $b=\diag(
\omega^{-1},\omega^{-1},\omega,\omega,\omega,\omega^{-1})$, Thus the Sylow $3$-subgroup $T$ of $C_X(t_2)$
contains two conjugates of $\langle e_3\rangle$, a conjugate of $\langle e_1\rangle$ and a conjugate of $\langle
e_2\rangle$. Now $C_V(a) =\langle w_1\wedge w_2\wedge w_5, w_3\wedge w_4\wedge w_6\rangle$ and $C_V(b) =\langle
w_1\wedge w_2\wedge w_6, w_3\wedge w_4\wedge w_5\rangle$ and so $C_V(T)=0$. It follows that $C_V(t_2)/[V,t_2]$
admits $C_X(t_2)$ as described in (i)(b).

There is a conjugate of $t_4$ which centralizes a subgroup isomorphic to $\Sp_4(2)$ in $C_X(t_1)/O_2(C_X(t_1))$.
By part (i)(a) $C_X(t_1)$ acts as  $\OO^-_6(2)$ on $[V,t_1]$ and $V/C_V(t_1)$ and naturally as $\SU_4(2)$ on
$C_V(t_1)/[V,t_1]$. Since $t_4$ is not a unitary transvection of $C_X(t_1)/O_2(C_X(t_1))$, we see that $\dim
[V,t_4] \geq 6$ and $[C_X(t_4), [V,t_4]]\not= 1$. Furthermore $\Sp_4(2)$ acts fixed point freely on
$C_W(t_4)/[W,t_4]$ for all $\U_4(2)$ sections in $V$. Therefore
 $\Sp_4(2)$ acts fixed point freely on $C_V(t_4)/[V,t_4]$. In particular  $|C_V(t_4)/[V,t_4]| = 2^{4x}$ where $x$ is some positive integer.
This shows that this module must be the 8-dimensional $\Sp_6(2)$-module and then we deduce $\dim C_V(t_4) = 14$.

We have $t_5 = t_4t_1$, and $C_X(t_5) \leq C_X(t_4)$.  As seen before we have that there is $U = \Sym(3) \times
\U_4(2)$ in $X$ such that as an $U$-module $V$ is a direct sum of the unitary module $V_2$ with a tensor product
of the 2-dimensional $\Sym(3)$- module with the $\OO^-_6(2)$-module. We may assume that $t_1 \in \Sym(3)$ and
$t_5$ and $t_4$ induce an outer automorphism on $\U_4(2)$. As $C_X(t_5)$ does not contain $\Sym(6) \times
\Sym(3)$, we see that $t_5$ acts faithfully on the normal $\Sym(3)$, while $t_4$ centralizes this group. We have
that $C_{V_2}(t_5)$ is of order 16. As $t_5$ inverts an element of order three in $\Sym(3)$, which acts fixed
point freely on $V_1$, we get that $C_{V_1}(t_5)$ is of order 64.
 Hence we have that $\dim C_V(t_5)= 10$.

For part (ii) we refer to Aschbacher \cite[7.5 (4)]{Aschbacher2E6} for centralizers of singular vectors in $V$.
This gives the centralizers of $v_1$, $v_2$ and  $v_3$.

 Let $Z= \langle e_3\rangle$, $Q= O_3(C_X(Z))\cong 3^{1+4}_+$ and set $U= C_V(Z)$. Then $\dim U= 2$ and $\dim [V,Z]=18$.
  Since $Q \le C_X(Z)'$, we have that $Q$ centralizes $U$. As none of
the singular vectors have such a subgroup centralizing them, we infer that the non-trivial elements of $U$ are
all non-singular. Now $U$ is normalized by $N_X(Z)$ and so we have that $C_X(U)$ has index at most $6$ in
$N_X(Z)$. By Lemma~\ref{U62J}, there is a conjugate $Y$ of $Z$ in $C_X(Z)$ which is not contained in $Q$. If
$[Y,U] = 1$, then $U = C_V(Y) = C_V(Z)$ and so $Y$ is conjugate to $Z$ in $N_X(U)$, which is not the case. Hence
$Y$ acts transitively on $U^\sharp$. This shows that $C_X(v_5)$ is as stated.

Let  $L\cong \L_3(4)$ be the Levi complement of the parabolic subgroup of $X$ which is the image of the
stabilizer of an isotropic $3$-space $I$ of the unitary space $W$ .Then $L$ also stabilizes an isotropic subspace
$J$ with $I \cap J=0$ and in fact $I$ and $J$ are the only such subspaces normalized by $L$. Now $L$ centralizes
$\langle i_1\wedge i_2\wedge i_3, j_1\wedge j_2\wedge j_3\rangle$ where $\{i_1,i_2,i_3\}$ and $\{j_1,j_2,j_3\}$
are bases for $I$ and $J$ respectively.

 Thus
by \ref{unique} $\dim C_V(L)= 2$ and this space is normalized by $\L_3(4):2$. It follows that this group
centralizes  at least one non-zero vector and this vector must be non-singular as none of the singular vectors
have such a stabilizer.  By \cite{Atlas} we have that $\L_3(4):2$ is a maximal subgroup in $F^\ast(X)$. Thus we
have at least two orbits of non-singular vectors and
 summing the lengths of these orbits we see that we have accounted for all the orbits of $X$ on $V$.
\end{proof}

\begin{lemma}\label{centralizerinvs}
Assume that $X \cong \U_6(2):2$ and that $V$ is a $20$-dimensional $\GF(2)X$-module. Let $Y$ be the semidirect product of $V$ and $X$. Then for $j$ an involution in $Y\setminus V$ we have one of the following:
\begin{enumerate}
\item $Vj$ is a $2$-central involution in $Y'/V$, $|C_V(j)|=2^{14}$ and
\begin{enumerate}
\item $C_{Y'}(j) \approx 2^{14}.2^{1+8}_+.\U_4(2)$;
\item $C_{Y'}(j) \approx 2^{14}.2^{1+8}_+.2^{1+4}.\Sym(3)$;
\item $C_{Y'}(j) \approx 2^{14}.2^{1+8}_+.3^{1+2}_+.\Q_8$;
\end{enumerate}
\item $Vj$ is not $2$-central in $Y'/V$ and $C_{Y'/V}(Vj) = 2^{4+8}.(\Sym(3)\times \Sym(3)) $,  $|C_V(j)|= 2^{12}$ and
\begin{enumerate}
\item $C_{Y'}(j) \approx 2^{12}.2^{4+8}.(\Sym(3)\times \Sym(3))$;
\item $C_{Y'}(j) \approx 2^{12}.2^{4+8}.\Sym(3)$;
\item $C_{Y'}(j) \approx 2^{12}.2^{4+8}.2^2$;
\end{enumerate}
\item $Vj$ is not $2$-central in $Y'/V$, $|C_V(j)|= 2^{10}$ and $C_{Y'}(j) \approx 2^{10}.2^{9}.3^2:\Q_8$;
\item  $j \in Y\setminus Y'$, $|C_V(j)|= 2^{14}$ and
\begin{enumerate}
\item $C_Y(j) \approx 2^{14}.(2 \times\Sp_6(2))$;
\item $C_Y(j) \approx 2^{14}.(2 \times 2^6.\L_3(2))$;
\item $C_Y(j) \approx 2^{14}.(2 \times \mathrm G_2(2))$; and
\end{enumerate}
\item  $j \in Y\setminus Y'$, $C_Y(j)\approx 2^{10}.(2 \times 2^5 .\Sym(6))$.
\end{enumerate}
\end{lemma}

\begin{proof} If $|C_V(j)| = 2^{10}$, then all involutions in $Vj$ are conjugate. Hence (iii) and (v) hold with Proposition \ref{Vaction}.

Let $j$ be 2-central. Then $C_V(j)/[V,j]$ is the $\U_4(2)$-module by Proposition \ref{Vaction} In particular we have three orbits of lengths 1,135, 120, which gives (i) (a) - (c).

If $j$ is as in (iv), then by Proposition \ref{Vaction} $C_X(j)$ induces on $C_V(j)/[V,j]$ the spin module and we have again orbits of lengths 1, 135 and 120, which gives (iv) (a) - (c).

Let finally $j$ be as in (ii). Then $|[V,j]| = 2^8$ and by Proposition \ref{Vaction} $C_V(j)/[V,j]$ is the $\OO^+_4(2)$-module for $C_X(j)$. Hence we have three orbits of lengths 1,6,9, which gives (ii) (a) - (c).
\end{proof}

\begin{lemma} \label{Vfacts} Suppose that $X \cong \U_6(2){:}2$ and that $V$ is an irreducible
$20$-dimensional $\GF(2)X$-module. Then $V$ is not a failure of factorization module.
\end{lemma}

\begin{proof} Suppose that $A \le P_1$ is an elementary abelian $2$-subgroup of $X$,  $|V:C_V(A)|\le |A|$ and
$[V,A,A]=0$. Then Lemma~\ref{nofours} and Proposition~\ref{Vaction}(i) imply that $$2^{8} \le |V:C_V(A)| \le |A|
\le 2^9$$ as the $2$-rank of $X$ is $9$. In particular, Proposition~\ref{Vaction} implies that all the
non-trivial elements of $A$ are conjugate to either $t_1$ or $t_2$. As the $2$-rank of $P_1/Q_1$ is $4$, $|A\cap
Q_1| \ge 2^4$. Since $t_1$ is weakly closed in $Q_1$ by Lemma~\ref{nofours}, there exist $b \in A\cap Q_1$
conjugate to $t_2$. Hence  $C_V(A)= C_V(b)\ge C_V(Q_1)$. Now $C_X(C_V(Q_1)) = Q_1$ by Proposition~\ref{Vaction}
and so $A \le Q_1$ which is absurd as $Q_1$ is extraspecial of order $2^9$.
\end{proof}

\begin{lemma}\label{contains a 2-central} Suppose that $X= \U_6(2){:}2$ and that $j\sim_Xt_2$.  Then every  normal subgroup of order
$8$ in  a Sylow $2$-subgroup of $C_X(j)$  contains a unitary transvection.
\end{lemma}

\begin{proof} By Lemma \ref{centralizerinvs}  we may assume that $P_1$ contains a Sylow 2-subgroup $T$
of $C_X(j)$  and $j \in Q_1$. Suppose that $A$ is a normal subgroup of  $T$  of order $8$ with $j\in A$. If $A
\cap C_{Q_1}(j) = \langle j \rangle$, then $[A, C_{Q_1}(j)] \leq \langle j \rangle$ and every non-trivial element
of $AQ_1/Q_1$ acts as a unitary transvection on $Q_1/\langle t_1 \rangle$. From \cite[Proposition~2.12
(viii)]{PS1}, we have $|A Q_1/Q_1|\le 2$ which means that $|A|\le 4$, a contradiction. Thus $A \cap C_{Q_1}(j)
\not \leq \langle j \rangle$. Since $C_{Q_1}(j)$ normalizes $A$ and $|Q_1:C_{Q_1}(j)|=2$, we now get $t_1 \in A$
and we are done.
\end{proof}

In the next lemma we present some results about the 10-dimensional Todd module for $\M_{22}$. A description of
this module may be found in \cite[Section 22]{Asch}. This module is seen to admit the action of $\Aut(\M_{22})$
and we continue to call this module the Todd module. We note that it is a quotient of the natural 22-dimensional
permutation module for $\Aut(\M_{22})$ (see \cite[(22.3)]{Asch}) and that the module is uniquely determined by this
property. The Todd module for $H=\L_3(4)$ is obtained as an irreducible 9-dimensional quotient
$\GF(2)$-permutation module obtained from the action of $H$ on the 21 points of the projective plane. Once
tensored with $\GF(4)$,  it can also be identified with the tensor product $N \otimes N^\sigma$ where $N$ is the
natural $\SL_3(4)$-module and $\sigma$ is the Frobenius automorphism. In particular, if $H_1$ and $H_2$ are the
two parabolic subgroups of $H$ containing a fixed Borel subgroup of $H$, then, without loss of generality, $H_1$
fixes a $1$-space and $O_2(H_2)$ centralizes a $4$-space one which $H_2/O_2(H_2)$ acts as an orthogonal module.

\begin{lemma}\label{M22} Let $X = \Aut(\M_{22})$, $Y= X'$ and  $V$ be the irreducible 10-dimensional Todd module for $X$ over $\GF(2)$.

\begin{enumerate}
\item If $x \in Y$ is an involution, then $|C_V(x)| = 2^6$.
\item Assume that $M \le X$ with $M \approx 2^4.\Sym(5)$ and $L= O_2(M)$, then  $L$ is elementary abelian of order 16 and $|C_V(L)| = 4$.
 \item Assume that $M \le X$ with $M \approx 2^4.\Alt(6)$ and $L= O_2(M)$, then  $L$ is elementary abelian of order 16, and $|C_V(L)| = 2^5$.
\item If $x \in X\setminus Y$ centralizes $M \approx 2^3.\L_3(2)$, then $|C_V(x)| = 2^7$ and involves two nontrivial $\L_3(2)$-modules.
\end{enumerate}
\end{lemma}

\begin{proof} From the   \cite[Table 5.3 c]{GLS3}, we have that  there is just one class of involutions in $Y=\M_{22}$.
Let $v$ be some vector in $V$ such that $|v^{X}| = 22$. Then $v$ is centralized by a subgroup $H \cong \L_3(4)$
and $V/\langle v \rangle$ is the Todd module \cite[(22.2) and (22.3.1)]{Asch}. Hence, by \cite[(22.2.1)]{Asch},
there is a parabolic subgroup  $H_1 \le H$ fixing a $1$-space in $V/\langle v \rangle$ such that, setting $E=
O_2(H_1)$, we have $H_1/E\cong \SL_2(4)$ and $E$ is elementary abelian of order $2^4$ admitting $H_1/E$ as
$\SL_2(4)$. It follows that $|C_V(E)| = 4$. Choose an involution $x \in H_1\setminus E$, then $x$ inverts some
element $\omega$ of order 5 with $|[V,\omega]| = 2^8$. Further $[C_V(\omega),x] = 1$. This shows $|C_V(x)| = 2^6$
and proves (i).

Let $H_2\le H$ be the companion parabolic subgroup to $H_1$, then, setting $E_2= O_2(H_2)$, we have $C_{V/\langle
v \rangle}(E_2)$ has dimension $4$. and it follows that $C_V(E_2)$ has dimension $5$.

 In $Y$ there is a subgroup $M \approx 2^4.\Alt(6)$ with $L=O_2(M)$ elementary abelian of order 16. As the
orbits of $Y$ on $V$ have length $22$, $231$ and $770$, we see that $M$ has no fixed point on $V$. Hence $E$ is
not normalized by $M$. Hence $N_X(E)\approx 2^4:\Sym(5)$ and we have (ii). Furthermore $E_1$ is
normalized by $M$ and so $E_1$ has to centralize the preimage of $C_{V/\langle v \rangle}(E_1)$ and we have
(iii).

Now let $x \in X\setminus Y$ be an involution, which centralizes $U \approx 2^3.\L_3(2)$ in $Y$. As
just elements from the orbit $v^{Y}$ are centralized by an element $\nu$ of
order 7, we see that $|C_V(\nu)| = 2$ and so $V$ involves three nontrivial $\L_3(2)$-modules. As $U$ is not a subgroup of
$\L_3(4)$, we see that $C_V(U) = 1$. In particular $\L_3(2)$ acts nontrivially on $[V,x]$. This now shows that $|[V,x]| = 8$ or $16$. In the second case we have that $|C_V(x)/[V,x]| = 4$ and so is centralized by an element of order 7, a contradiction. This
 shows  (iv).
\end{proof}

Our next lemma of this section requires the following transfer theorem.

\begin{theorem}\label{transfer} Let $M$ be a subgroup of a finite group $G$ with $G = O^2(G)$,
$|G:M|$ odd and $M > O^2(M)M^\prime$. Suppose that $E$ is an elementary abelian subgroup of a Sylow $2$-subgroup
$T$ of $M$ such that $E$ is weakly closed in $T$ and $N_G(E) \leq M$. Let $T_1$ be a maximal subgroup of $T$ with
$|M : O^2(M)T_1| = 2$. Then there exists $g \in G\setminus M$ such that $|E^g : E^g \cap M| \leq 2$ and $E^g \cap
M \not\leq O^2(M)T_1$.
\end{theorem}

\begin{proof} This is \cite[Theorem 2.11 (i)]{SoWo}.
\end{proof}

\begin{lemma}\label{autm22} Suppose that $G$ is a group, $M$ is a $2$-local subgroup of $G$ with
$F^*(M)=O_2(M)$. Assume that
 $M/O_2(M) \cong \Aut(\M_{22})$, $O_2(M)$ is elementary abelian of order $2^{10}$ and $O_2(M)$ is the Todd module for
 $M/O_2(M)$.
  Then
  \begin{enumerate}
\item For involutions $x$ in $M \setminus O^2(M)$, the $2$-rank of $C_{M}(x)$ is at most $ 8$; and
\item $G$ has a subgroup of index $2$.
\end{enumerate}
\end{lemma}

\begin{proof} Let $E= O_2(M)$, $X= M/E$  and $Y= X'$. From \cite[Table 5.3 c]{GLS3} we see that  $X$  has exactly two conjugacy classes of involutions not in $Y$ one with centralizer of shape $2 \times  2^3:\L_3(2)$ and the other with centralizer $2\times 2^4:(5:4)$. Also by \cite[Table 5.3 c]{GLS3}, the normalizer of a Sylow $11$-subgroup of  $Y$ has order $55$. Hence one class of involutions in $X\setminus Y$ contains elements which normalize, and consequently invert,  a Sylow $11$-subgroup. Furthermore, such an involution commutes with an element of order $5$.

Aiming for a contradiction, let $x \in N_G(E)$ with $Ex\not\in X$   and  $F \leq C_{M}(x)$ with $F$ is
elementary abelian of order at least $2^9$. Since the $2$-rank of $X$ is $5$, we have $|C_E(F)|\ge 2^4$.

If  $Ex$  inverts an element of order 11 in $X$, then
 $|C_E(x)| = 2^5$  and $C_X(Ex) \cong  2\times (2^4:(5:4))$.  Let $L= O_2(C_Y(Ex))$.
  By Lemma \ref{M22} (ii),  we have that $|C_E(L)| \leq 2^2$.  Since the involutions which invert an element of order $5$ in $C_X(Ex)$   can only centralize $2^3$ in $C_E(x)$, we infer that $FE/E \leq L$.
    If $F$ centralizes $C_E(x)$ then the normal closure of  $FE/E$ in  $C_{M/E}(Ex)$ also is abelian and so we may assume that $FE/E = L$ in this case. On the other hand, if $F$ does not centralize $C_E(x)$, then $|FE/E|\ge 2^5$ and we also have $FE/E = L$.
Hence in any case $FE/E = L$. However this implies that  $|F| \leq 2^7$ as $|C_E(L)| \leq 4$ and is a contradiction. Hence $F$ contains no such involutions.

So we have  $C_X( Ex) \cong 2 \times 2^3:\L_3(2)$. Let $L=O_2(C_Y( Ex))$
and $L_1 \le C_X(Ex)$  be such that $L_1 \cong \L_3(2)$. Let $e \in L_1$ be an involution. Then
$Le$ contains  representatives of two $LL_1$-conjugacy classes of involutions. As $x$ is not 2-central in $X$, we have that $x \sim_X x\ell$ for
some $1 \not= \ell \in L$. It follows that all the involutions in $Lx$ are conjugate to $x$ in $X$. Hence
we see that the coset $Lex$ contains an involution  which is not conjugate to $x$ in $X$.

Assume  that $(F \cap T)E/E \not\leq L$. Let $e \in FE/E \cap L_1L \setminus L$.
 If $|(FE/E) \cap L| > 2$ then $(FE/E \cap L)ex$ is the set of involutions in $Lex$. But this coset contains an
 involution which inverts an element of order 11 and we have already seen that such elements cannot be in $F$.
 So  $|(FE/E) \cap L| \leq 2$ and consequently  $|FE/E| \leq 16$. By Lemma \ref{M22} (iv),  $|C_E(x)| = 2^7$ and, for $e \in FE/E \setminus L\langle Ex\rangle$, as $C_E(x)$ has two non-trivial $3$-dimensional composition factors for $L_1$,
$|C_{E}(x) : C_{C_E(x)}(e)| \geq 4$.
Therefore  $|C_E(F)| =
2^5$ and $|FE/E| = 2^4$.  In $L_1$ there are two conjugacy classes of fours groups. One which is contained in an elementary
abelian group of order $2^5$ in $M/E$ and one which is contained in a conjugate of $O_2(C_{M/E}(x))$.
If
$FE/E$ is contained in an elementary abelian group $F_1$ of order $2^5$ in $\Aut(\M_{22})$, then, as $|C_E(F)| = 2^5$,
we get that $|C_E(F_1)| \geq 2^3$, which contradicts Lemma \ref{M22} (ii). Therefore  $FE/E$ is uniquely determined and
is conjugate to $\langle L, Ex\rangle$ in $M/E$. In particular $|C_E(\langle L, Ex\rangle)| =
2^5$.  But then $L_1$ cannot induce two non-trivial irreducible modules in $C_E(x)$, which contradicts Lemma
\ref{M22}(iv).

Suppose that $w \in Lx$ and let $L_w= O_2(C_Y(w))$.  We have that $C_{LL_1}(w)/L$ is a parabolic subgroup of $LL_1/L$. Therefore
$LL_w$ has order $2^5$ and consequently $L \cap L_w$ has order $2$.  Now we have  $(F\cap Y)E/E \cap L \cap L_w$ which means that  $|FE/E| \le 2^2$ and $|C_E(F)|\ge 2^7$. Using Lemma~\ref{M22} , for $f \in O^2(M) \setminus E$, we have
that $|C_E(f)| = 2^6$. Hence $|FE/E|=2$ and $|C_E(F)|=2^8$ contrary to Lemma~\ref{M22} (iv). This proves (i).

We recall that $V$ is not a failure of factorization module for $X$. Thus, for $S \in \syl_2(M)$, $E= J(S)$ and hence $E$ is weakly closed in $S$ with respect to $G$. In particular, as $M= N_G(E)$, $S \in \Syl_2(G)$ and  $M$ has odd index in $G$. Therefore (ii) follows from Theorem~\ref{transfer} and part (i).
\end{proof}

\begin{lemma}\label{Fusion}
Suppose that $G$ is a group, $E$ is an extraspecial subgroup of $G$,  $H=N_G(E)= N_G(Z(E))$, $C_G(E)= Z(E)$ and $S\in \syl_p(H) \subseteq \syl_p(G)$. Assume that if $g \in G$ and  $Z^g \le E$ then every element of  $Z^gZ$ is conjugate to an element of $Z$ and assume that no element of $S\setminus E$  centralizes a subgroup of index $p$ in $E$. Then, for all $d \in E$ with $d^G \cap Z = \emptyset$,  $\Syl_p(C_H(d)) \subseteq \syl_p(C_G(d))$  and  $d^G \cap E = d^H$.
\end{lemma}
\begin{proof}  Assume that $d \in E$ is not $G$-conjugate to an element of $Z$. Let $T \in \Syl_p(C_G(d))$. Then $Z(T)$ centralizes $C_E(d)$ which has index $p$ in $E$. Thus $Z(T) \le E$  and so $Z(T) = Z(C_E(d)) = \langle d \rangle Z$. In particular, $Z$ is the unique $G$-conjugate of $Z$ contained in $\langle d \rangle Z$. Therefore $N_G(T)  \le H$ and consequently $T \in \Syl_p(C_G(d))$.

Now assume that $e=d^g \in d^G \cap E$ and let $R \in \Syl_p(C_H(e))$. Then, as $T^g \in \Syl_p(C_G(e))$, there exists $h \in C_G(e)$ such that $T^{gw}= R$. But then $Z\langle s \rangle ^{gw} = Z\langle e \rangle$ and as $Z$ is the unique conjugate of $Z$ in $Z\langle e \rangle$ we conclude that $Z^{gw}= Z$.  Thus $gw \in H$ and $d^{gw}= e^w=e$.  Thus $d^G \cap E = d^H$ as claimed.
\end{proof}

\begin{lemma}\label{ptransfer} Suppose that $p$ is a prime, $G$ is a group and $P \in \Syl_p(G)$.
Assume that $J=J(P)$ is the Thompson subgroup of $P$. Assume that $J$ is elementary abelian. Then
\begin{enumerate}
\item $N_G(J)$ controls $G$-fusion in $J$; and
 \item if  $J \not \le
N_G(J)'$, then $J \not \le G'$.
\end{enumerate}
\end{lemma}

\begin{proof} Part (i) is well-known see \cite[37.6]{Asch}. Part (ii) is proved in \cite[Lemma 2.2(iii)]{PS1}.\end{proof}

The next lemma is a straightforward consequence of Goldschmidt's Theorem on groups with a strongly closed abelian
subgroup \cite{Goldschmidt}. Recall that for subgroups $A \le H \le G$, we say that $A$ is \emph{weakly closed} in $H$ with respect to $G$ provided that
for $g \in G$, $A^g \le H$ implies that $A^g= A$. We say that $A$ is \emph{strongly closed} in $H$ with respect to $G$ so long as, for all
$g \in G$,  $A^g \cap H \le A$.

\begin{lemma}\label{Gold} Suppose that $K$ is a group, $O_{2'}(K)=1$,  $E$ is an abelian $2$-subgroup of $K$ and $E$ is strongly closed in $N_K(E)$. Assume that $F^*(N_K(E)/C_K(E))$ is a non-abelian simple group.
 Then  $K= N_K(E)$.
\end{lemma}

\begin{proof} See \cite[Lemma~2.15]{F4}. \end{proof}
 We will also need the following statement of Holt's Theorem \cite{Ho}.

\begin{lemma}\label{Holt}
Suppose that $K$ is a simple group,  $P$ is a  proper subgroup of $K$ and  $r$ is a $2$-central element of $K$.
If $r^K\cap P= r^P$ and $C_K(r)\le P$, then  $K\cong \PSL_2(2^a)$ ($a \ge 2$), $\mathrm{PSU}_3(2^a)$ ($a \ge 2$),
${}^2\B_2(2^a)$ ($a\ge 3$ and odd) or $\Alt(n)$ where
in the first three cases $P$ is a Borel subgroup of $K$ and in the last case $P \cong \Alt(n-1)$.
\end{lemma}

\begin{proof} This is \cite[Lemma 2.16]{F4}.
\end{proof}

\begin{definition} \label{U6F4def}We say that  $X$ is similar to a $3$-centralizer in a group of type  $\U_6(2)$ or $\F_4(2)$
provided the following conditions hold.
\begin{enumerate}
\item $Q=F^*(X)$ is extraspecial of order $3^5$; and
\item $X/Q$ contains a normal subgroup isomorphic to $\Q_8 \times \Q_8$.
\end{enumerate}
\end{definition}

The main theorems  of \cite{PS1, F4} combine to give the following result which is also recorded  in \cite{F4}.

\begin{theorem}\label{F4U6Thm} Suppose that $G$ is a group, $Z \le G$ has order $3$ and set $M = C_G(Z)$.
If $M$ is similar to a $3$-centralizer of a  group of type $\U_6(2)$ or $\F_4(2)$ and $Z$ is not weakly closed in
a Sylow $3$-subgroup of $G$  with respect to $G$, then  either $F^*(G) \cong \U_6(2)$ or  $F^*(G) \cong \F_4(2)$. Furthermore, if $F^*(G) \cong \U_6(2)$, then $Z$ is weakly closed in $O_3(M)$ with respect to $G$ and if $F^*(G) \cong \F_4(2)$, then $Z$ is not weakly closed in $O_3(M)$ with respect to $G$.
\end{theorem}

\begin{definition}\label{O8def}
We say that  $X$ is similar to a $3$-centralizer in a group of type  $\Aut(\Omega_8^+(2))$
provided the following conditions hold.
\begin{enumerate}
\item $Q=F^*(X)$ is extraspecial of order $3^5$;
\item $X/Q \cong \SL_2(3)$ or $\SL_2(3) \times 2$;
\item $[Q,O_{3,2}(X)]$ has order $27$.
\end{enumerate}
\end{definition}

\begin{theorem}[Astill \cite{Astill}]\label{Astill} Suppose that $G$ is a group, $Z \le G$ has order $3$ and set $M = C_G(Z)$.
If $M$ is similar to a $3$-centralizer of a  group of type $\Aut(\Omega_8^+(2))$  and $Z$ is not weakly closed in $O_3(C_G(Z))$  with respect to $G$, then  either $G \cong \mathrm \Omega_8^+(2):3$ or  $F^*(G) \cong \Aut(\Omega_8^+(2))$.
\end{theorem}

\section{Strong closure}

The main result of this section will be used in the final determination of the centralizer of an involution in
${}^2\E_6(2)$. Remember that for a prime $p$ and a group $X$  a subgroup $Y$ of order divisible by $p$ is \emph{strongly $p$-embedded} in $X$ so long as $Y\cap Y^g $ has order coprime to $p$ for all $g \in X \setminus Y$.

\begin{lemma}\label{fusion2} Suppose that $p$ is a prime, $X$ is a group and $H$ is strongly $p$-embedded in $X$.
If $x\in H$,  $y \in x^X\cap H$ and $p$ divides both  $|C_H(x)|$ and $|C_H(y)|$, then $y \in x^H$.
\end{lemma}

\begin{proof}
Since $H$ is strongly $p$-embedded in $X$ and $p$ divides $|C_H(x)|$, $C_{H}(x)$ contains a Sylow $p$-subgroup $P$ of
$C_X(x)$. Let $g \in X$ be such that $y^g = x$. Since $p$ divides $|C_H(y)|$ there is an element  $d \in C_H(y)$  of
order $p$. Then $d^{g}$ is a $p$-element of $C_H(x)$ and hence there exists an element  $w\in C_G(x)$ such that
$d^{gw} \in P$. Then, as $H$ controls $p$-fusion in $X$ (\cite[Prop. 17.11]{GLS2}), there exists  $h\in H$ such that $d= d^{gwh}$.  As $H$ is strongly $p$-embedded in $G$, we now have $gwh \in C_X(d) \le H$. Hence $gw \in H$, and $$y^{gw}= x^{w} = x$$ as claimed.
\end{proof}

\begin{lemma}\label{centinH} Suppose that $X$ is a group, $H = N_X(A) $  with $H/A \cong   \U_6(2)$ or $\U_6(2):2$, $|A|=2^{20}$ and $A$
a minimal normal subgroup of $H$. Then $C_H(x)$ contains a Sylow $2$-subgroup of $C_X(x)$ for all {$x \in A$}.
\end{lemma}

\begin{proof} Let $S \in \Syl_2(C_X(x))$ with $S \cap H \in \Syl_2(C_H(x))$. As,
by Proposition~\ref{Vfacts} (i), $A$ is not a failure of factorization module for $H/A$, we have $A = J(S\cap H)$
{from \cite[Lemma 26.7]{GLS2}}.  In particular, we have $N_S(S\cap H) \le N_G(J(S\cap H)) = H$. Hence $S =S\cap H$.
\end{proof}

We can now prove Theorem~\ref{Not3embedded} which we restate for the convenience of the reader.

\begin{theorem}\label{Not3embedded1} Suppose that $X$ is a group, $O_{2'}(X)= 1$, $H = N_X(A)= AK$  with $H/A \cong K \cong \U_6(2)$ or $\U_6(2):2$, $|A|=2^{20}$ and $A$
a minimal normal subgroup of $H$. Then $H$ is not a strongly $3$-embedded subgroup of $X$.
\end{theorem}

\begin{proof}  Suppose that $H$ is strongly $3$-embedded in $X$. Let $S \in \Syl_2(H)$. Then  Lemma \ref{centinH} yields $S \in \Syl_2(X)$.
 We now claim that $A$ is strongly closed in $H$ with respect to $X$. Assume that, on the contrary, there
is $u \in A$, $g \in X$ and $v \in H\setminus A$ with $v^g = u$. If $3$ divides both  $|C_H(u)|$ and $|C_H(v)|$,
then $u$ and $v$ are $H$-conjugate by Lemma~\ref{fusion2}. Since $A$ is normal in $H$, this is impossible.
{Therefore, as $H= AK$ is a split extension, Proposition~\ref{Vaction} and Lemma~\ref{centralizerinvs} together, imply that there is} a unique
possibility for the conjugacy class of $v$ in $H$ and $C_{S}(v)A/A$ has index $2$ in $S/A$. In addition, we  have
$|C_{A}( v)|= 2^{12}$.

Since $v \in A^{g^{-1}}$, there exists a Sylow $2$-subgroup $T$ of $C_X(v) $ which contains both $C_S(v)$ and a
conjugate of $A$ which contains $v$. Let $A_v = J(T)$. If $C_A(v) \le A_v$, then, as $[A,v]\le C_A(v)$,
$\langle A,A_v\rangle$ normalizes $\langle v,A\cap A_v\rangle$. Because $A$ is the Thompson subgroup of any $2$-group
which contains $A$,  $A$ and $A_v$ are conjugate in $\langle A,A_v\rangle$. But $A$ does not centralize $\langle v, A_v \cap A \rangle$ while $A_v$ does,
 which is  a contradiction. Thus $C_A(v) \not \le A_v$.

We  have $(A_v \cap C_S(v))A/A$ is an elementary abelian normal subgroup of $C_S(v)A/A$ and, as $(A_v \cap
C_S(v))A/A$ only contains elements which are conjugate to $Av$, we have $|(A_v \cap C_S(v))A/A|\leq 4$ from
Lemma~\ref{contains a 2-central}.  Combining this with  the fact that $A_v \cap C_S(v) \cap A < C_A(v)$, we
deduce that $|A_v\cap C_S(v)|\le 2^{13}$. In particular we have that $|T : A_vC_S(v)| \le 4$. {Now using Lemma
\ref{centinH} and Proposition~\ref{Vaction}    we see that $v$ is $H^{g^{-1}}$-conjugate to an element in $A_v$
in class $v_1$ or $v_2$ (using the notation as in Proposition \ref{Vaction}). Furthermore,} $v$ is a singular
element. Suppose that $v$ is conjugate to $v_2$. Then  $|T : A_vC_S(v)| = 4$ and so $|A_v \cap C_S(v)| = 2^{13}$.
But any subgroup of $A_v$ of order $2^{13}$ is generated by non-singular vectors, and {as we have seen such
elements are not conjugate} to elements in $H \setminus A$, a contradiction. So we have that $v$ {is conjugate
to} $v_1$. Now let  $T$ be a Sylow 2-subgroup of $C_X(v)$, which contains $A_vC_S(v)$. Then $T \in \Syl_2(X)$ {by
Lemma~\ref{centinH}}. {Once again, as $A_v \cap C_S(v)$ is not generated by non-singular vectors,} we get that
$|A_v \cap C_S(v)| \le 2^{12}$ and so $|T : A_vC_S(v)| \le 2$. Further we have $|C_S(v) \cap A_v| \ge 2^{11}$.
Therefore, {as there are only $891$ conjugates of $v$ in $A_v$}, $|(A_v \cap C_S(v))\setminus A| \le 891$. It
follows that $|A \cap A_v| \le 2^9$. Since $|(C_S(v) \cap A_v)A/A|\le 2^2$, we get  $|A \cap A_v| = 2^9$ and
$|C_S(v) \cap A_v| = 2^{11}$. But then $891 \ge |(A_v \cap C_S(v))\setminus A| = 1536$ which is a contradiction.
Hence $A$ is strongly closed in $H$.

Since $A$ is strongly closed in $H$ and $O_{2'}(X)=1$, we now have that $X= H$ by Lemma~\ref{Gold} and this is
impossible as $H$ is strongly $3$-embedded. This completes the proof of the theorem.
\end{proof}

\section{The Structure of $H$}

From here on we assume that $G$ satisfies the hypothesis of Theorem~\ref{Main} or Theorem~\ref{Main1}. We let
$H \le G$ be a subgroup of $G$ which is similar to the $3$-centralizer in a group of type ${}^2\E_6(2)$ or $\M(22)$. We let $Z= Z(O_3(F^*(H)))$ and assume that $H = C_G(Z)$.

We will use the following notation  $Q= O_3(H)$, $S \in \Syl_3(H)$ and $Z =\langle z\rangle= Z(S)$. We select $R
\in \Syl_2(O_{3,2}(H))$ such that  $S=N_S(R)Q$. Then $R$ is isomorphic to a subgroup of $\Q_8\times \Q_8\times
\Q_8$ containing the centre of this group and of order  $2^7$ when $H$ has type $\M(22)$ and order $2^9$ when
$H$ has type ${}^2\E_6(2)$. Note that $\Omega_1(R) $ is elementary abelian of order $2^3$.  For $i=1,2,3$, let
$\langle r_i\rangle \le \Omega_1(R)$ be chosen so that $C_Q(r_i)$ is extraspecial of order $3^5$. {We  set, for
$i=1,2,3$, $Q_i= [Q,r_i]$ and note that $Q_i$ is extraspecial of order $3^3$.}

{If $|R|= 2^9$, we let $R_1$, $R_2$ and $R_3$ be the three normal subgroups of $R$ which are isomorphic to
$\Q_8$ such that $[R_i,Q] = Q_i$. Notice  that we have  $Z(R_i) = \langle r_i\rangle$ in this case.}
Further we set $B = C_S(\Omega_1(Z(R)))$.

\begin{lemma} We have $Q_1 \cong Q_2 \cong Q_3 \cong 3^{1+2}_+$ and that pairwise these subgroups commute. \end{lemma}

{\begin{proof}  This follows from the Three Subgroup Lemma and the definitions of $r_i$ and $Q_i$.
\end{proof}
}

Since each $Q_i$ has exponent $3$, $Q$ has exponent $3$ and so $\Out(Q) \cong \GSp_6(3)$. For later calculations,
for each $i= 1, 2,3 $, we select $q_i, \wt q_i \in Q_i$ such that  $[q_i,S] \le Z$
 \begin{center}$q_i^{r_i}= q_i^{-1}$, $\wt q_i^{r_i}= \wt q_i^{-1}$ and $[q_i,\wt q_i]=z$.
\end{center}



 We set $\ov H = H/Q$. Then the following lemma  follows from
the structure of $\GSp_6(3)$ and the definition of the $3$-centralizers in groups of type $\M(22)$ and
${}^2\E_6(2)$.

\begin{lemma}\label{H/Q struct}
 We have $\ov R $ is normal in $ \ov H$ and, in particular, $\ov H$ is isomorphic to a subgroup of
$\Sp_2(3)\wr \Sym(3)$ preserving the symplectic form.

\end{lemma}

\begin{proof} {This follows from the definition of $H$. Note also that $\ov H$ preserves the ``perpendicular" decomposition of
$Q$ as the central product of $Q_1$, $Q_2$ and $Q_3$.}
\end{proof}

If the Sylow $3$-subgroup $S$ of $H$ equal $ Q$, then, as $Z$ is not weakly closed in $S$ by hypothesis, there
exists $g \in G$ such that $Z^g \le S=Q$ and $Z\neq Z^g$. Now $C_S(Z^g) \cong 3 \times 3^{1+4}_+$ and so
$C_Q(Z^g)'= Z$. However, $C_G(Z^g)$ is $3$-closed with Sylow $3$-subgroup $Q^g$ and derived subgroup $Z^g$.
Therefore we have

\begin{lemma}\label{S>Q} $S > Q$.
\end{lemma}

We draw further information about the structure of $\ov S$ from Lemma~\ref{H/Q struct}.

{
\begin{lemma}\label{added} The following hold:
\begin{enumerate}
\item $\ov S$ is isomorphic to a subgroup of $3\wr 3$ and $|S:BQ|\le 3$;
\item if $x \in S\setminus BQ$ has order $3$, then $|C_{Q/Z}(x)|= 9$, $|[Q/Z,x]| = 3^4$  and the preimage of
$C_{Q/Z}(x)$ is equal to the centre of $[Q,x]$;
\item if $x\in  BQ$, then $|C_{Q/Z}(x)| \ge 3^3$;
\item if $\ov S$ contains $\ov E$ of order $9$ with $\ov S= \ov E \ov B$, then $|C_{Q/Z}(\ov E)| = 3$; and
\item if $\ov F \le \ov S$ is elementary abelian of order $27$, then $\ov F = \ov B$.
\end{enumerate}
\end{lemma}

\begin{proof} Lemma~\ref{H/Q struct} (i) implies  that $\ov S$ is isomorphic to a subgroup of  the wreath product $3\wr
3$  and, as by design,  $\ov B$ is the intersection of $\ov S$ with  the base group of this group,  (i) holds.

Assume that $x \in S\setminus BQ$. Since  $x\not\in BQ$, $x$ permutes the set $\{Q_1,Q_2,Q_3\}$ transitively  and
therefore $Q/Z$ is a sum of two regular representations of $\langle x\rangle$. It follows that  $[Q/Z,x]$ has
order $81$, $|C_{Q/Z}(x)|$ has order $9$ and $C_{Q/Z}(x)= [Q/Z,x,x]$. Let $J$ be the preimage of $C_{Q/Z}(x)$.
Then $[J,x,Q]=1$ and $[J,Q,x]=1$. Hence the Three Subgroup Lemma implies that $J\le Z([Q,x])$ and as $Q$ is
extraspecial, equality follows.

Part(iii) follows from the fact that $BQ$ normalizes each $Q_i$, $1\le i \le 3$.

For part (iv), we have $\ov E$ contains an element which acts nontrivially on each of $Q_i$, $i = 1,2,3$, and a
further element which permutes the $Q_i$ transitively. So the result follows.

Finally (v) follows from (i) as $3\wr 3$ contains a unique elementary abelian subgroup of order $27$.

\end{proof}
}

The next lemma shows that $Z$ is not weakly closed in $Q$. As we will see this is not an immediate observation.

\begin{lemma}\label{ZnotweakQ} $Z$ is not weakly closed in $Q$ with respect to $G$.
\end{lemma}

\begin{proof} Assume that $Z$ is weakly closed in $Q$.
By hypothesis we have that $Z$ is not weakly closed in $S$ with respect to $G$.
 Hence {there exists $g\in G$ such that}  $Y=Z^g \le S$ and $Y\not\le Q$.\\

\begin{claim}\label{weak1} We have $Y \le BQ$.
\end{claim}
\\

Suppose that $Y  \not \le BQ$. Then, by Lemma~\ref{H/Q struct}, $\ov Y $ permutes the set $\{Q_1,Q_2,Q_3\}$
transitively and $\ov Y$ centralizes $\ov f = \ov {r_1r_2r_3}$ which has order $2$. Furthermore by
Lemma~\ref{added} (ii), $[Q/Z,Y]/C_{Q/Z}(Y)$ and  $C_{Q/Z}(Y)$ have order $9$.  In particular, every element of
order $3$ in $Qz^g$ is conjugate to an element of $Zz^g$. Therefore, as $Z$ normalizes $R$, we may assume that
$Y$ normalizes $R$ and so we  can further assume that $f=r_1r_2r_3 \in C_R(Y)$.

{Let $J$ be the preimage of $C_{Q/Z}(Y)$ and set $E = [J,f]$. Then, as $J$ is abelian by Lemma~\ref{added} (ii),
$E$ has order $9$ and is centralized by $Y$. Hence $J= C_Q(Y)=ZE$. Furthermore, Lemma~\ref{added} (ii) shows that
$[Q,Y] = C_Q(E)$. Since $[Y,f]=1$ and $[C_Q(E),f]=[Q,Y,f]=[Q,Y]$, the Three Subgroup Lemma (to get the second
equality) implies}
$$[Q,Y,Y] = [C_Q(E),f, Y]=[C_Q(E),Y,f]=[Q,Y,Y,f]=E.$$
In particular, if $y = z^g$, then every element of the coset $Ey$ is conjugate to $z$. Hence $Ey \cap Q^g
\subseteq  \{y,y^{-1}\}$ as $y^G \cap Q^g \subseteq \{y, y^{-1}\}$. Thus $E \cap Q^g= 1$. As $f$ inverts $Q \cap
Q^g$ we have that $Q \cap Q^g \leq E$ and so $Q \cap Q^g = 1$. Since $ZE\le C_G(Y)$, we now have $ZEQ^g/Q^g$ is
elementary abelian of order $3^3$. It follows from Lemma~\ref{added} (v) that $Z$ centralizes
$\Omega_1(R^g)Q^g/Q^g$. Hence $|C_{Q^g/Y}(Z)| \ge 3^3$ {by Lemma~\ref{added}(iii)}. Now we have that
$|\ov{C_{Q^g}(Z)}| \geq 3^3$. Since $\ov Y$ centralizes $\ov{C_{Q^g}(Z)}$ this is impossible. Hence
 \ref{weak1} holds.\qedc


Reiterating the statement of  \ref{weak1}, we have $z^G \cap H \subseteq BRQ$.

\begin{claim}\label{cent} We have that $C_Q(Y)$ does not contain a subgroup $F$ isomorphic to $3^2 \times 3^{1+2}_+$.
\end{claim}
\medskip

Suppose false and assume that $F$ is such a subgroup. As
 $Z \not \le Q^g$, we have that $FQ^g/Q^g$ is isomorphic to $3^{1+2}_+$. Since $F$ centralizes $F \cap Q^g$ which
 has order $9$, we have a contradiction to the fact that $|C_{Q^g/Y}(F)|=3$, see Lemma \ref{added} (iv).
 \qedc

\begin{claim}\label{weak2} For $\{i,j\} \subset\{1,2,3\}$ with $i \neq j$,   $[Y,Q_iQ_j] \not \le Z$.
\end{claim}

\medskip

 Assume that $[Y,Q_iQ_j]  \le Z$. Then $C_{Q/Z}(Y)$ has order $3^5$ and, letting $E_1$ be its preimage, we have
 $E_1 \cong 3 \times 3^{1+4}_+$. If $E_1$ is centralized by $Y$, then $E_1Q^g/Q^g$ must be elementary abelian and we have
 $Z \le Q^g$ which is a contradiction. So suppose that $[Y,E_1]=Z$. Then $E_2=C_{E_1}(Y)\cong 3^2 \times 3^{1+2}_+$. But this contradicts \ref{cent}.
\qedc

\begin{claim}\label{weak3} If $E \le C_Q(Y)$ with $|E|= 27$,  then the non-trivial cyclic subgroups contained in  $EY$
but not in $E$ are not all conjugate to $Z$.
\end{claim}\\

\medskip
Suppose that every non-trivial cyclic subgroup $EY$ not contained in $E$ is conjugate to $Z$. Then
$E \cap Q^g = 1$ for otherwise $(E \cap Q^g)Y\le Q^g$ contains a conjugate of $Z$. Thus \ref{weak1} implies that
$EY \le B^{gh} Q^g$ for some appropriate $h \in H^g$. But then there is a subgroup $U \le EY$, $U \not= Y$ such that $U$
is $G$-conjugate to $Z$ and such that $U$ centralizes $(Q_1Q_2)^{gh}/Y$. This violates \ref{weak2}.
\qedc

\begin{claim}\label{weak4} There are non-trivial cyclic subgroups of $YZ$ which are not conjugate to $Z$.
 In particular, $C_{Q}(Y)/Z = C_{Q/Z}(Y)$.
\end{claim}

\medskip

Suppose that statement is false. Let the subgroups of order $3$ in $YZ$ be $Y_1$, $Y_2$, $ Y$ and $Z$. Then by assumption all these groups are $G$-conjugate to $Z$. Let $E= [Q,Y]Z$. Then the cyclic subgroups of $EY$ not contained in $E$ are $Y_1^Q\cup Y_2^Q \cup Y^Q$. Since $|E| \ge 27$ by \ref{weak1} and \ref{weak2}
we have a contradiction to \ref{weak3}. Let $C$ be the preimage of $C_{Q/Z}(Y)$. Then, as $Y$ and $Z$ are the only
$G$-conjugates of $Z$ in $YZ$, $C$ centralizes $Y$ and have $C= C_Q(Y)$.\qedc

\begin{claim}\label{weak5} $C_Q(Y)$ is elementary abelian of order $81$. In particular, for $i=1,2,3$, $[Q_i,Y] \not\le Z$.
\end{claim}
\medskip

Otherwise $Y$ centralizes $Q_1/Z$ say and then $C_{Q}(Y) \cong 3^2\times 3^{1+2}_+$ by \ref{weak4}. Now
\ref{cent} gives a contradiction. \qedc

{ Since $[Q,Y] = C_Q(Y)$, every subgroup of $[Q,Y]Y$ order $9$ containing $Z$ is $Q$-conjugate to $YZ$.  As
$[Q,Y]Y = C_{QY}(C_Q(Y))$ is normalized by $\Omega_1(R)$,  we may suppose that $[\Omega_1(R),ZY] = 1$.
 From \ref{weak5} we have $|C_{Q^g}(Z)/Y|=3^3$  and so Thompson's $A\times B$ Lemma \cite[Lemma 11.7]{GLS2} implies that $\Omega_1(R)$ is isomorphic to a subgroup of $\GL_3(3)$. Since
 all elementary abelian subgroups of order $2^3$ in $\GL_3(3)$ contain the centre of $\GL_3(3)$, there exists $x \in \Omega_1(R)$ such that $C_{Q^g}(Z)/Y$ is inverted by $x$.  Hence $C_{Q^g}(Z)= Y[C_{Q^g}(Z),x]$. Because  $\ov{C_{Q^g}(Z)}$ normalizes, and is normalized by, $\Omega_1(R)$, we have $$Q \ge [C_{Q^g}(Z),\Omega_1(R)]=[C_{Q^g}(Z),x].$$
Therefore $C_{Q^g}(Z)Q= YQ$ and $|C_{Q^g}(Z)\cap Q|=  |Q\cap Q^g|=3^3$.}

%
%
%

 %
%

{Set $D= Q\cap Q^g$ and  $U= ZDY$.  Then $U$ is elementary abelian of order $3^5$.} Let $P= \langle Q,
Q^g\rangle$ and note that $P$ normalizes $U$. Since $Z$ is the only $G$-conjugate of $Z$ in $DZ$ and $P$ does not
normalize $Z$, we see that there are $P$-conjugates of $Z$ which are not contained in $DZ$. Now conjugating by
$Q$, we see that there are $28$, $55$ or $82$ $P$-conjugates of $Z$ in $U$. Since $7$ and $41$ do not divide
$|\GL_5(3)|$, we have that there are exactly $55$  $P$-conjugates of $Z$ in $U$. Similarly, there are $55$
$P$-conjugates of $Y$ and so we infer that $Z$ and $Y$ are $P$-conjugate. Since $DZ$ and $DY$ each only have one
$G$-conjugate of $Z$, we have that $U\setminus( DZ\cup DY)$ contains at most two elements which are not conjugate
into $Z$.  Since $Q$ does not normalize $Y$ and does normalize $DZ$, there is a $u \in P$ with $(ZD)^{u} \not
\subseteq DZ \cup DY$ .
 Set $D_1 = D \cap (DZ)^{u}$.
Then $|D_1| \geq 9$. Choose $x \in (DZ)^{u} \setminus (ZD \cup DY)$. Then in { $\langle D_1, x \rangle$ there are
nine subgroups} of order three not in $ZD \cup DY$, in particular at least eight of them are conjugate to $Z$,
which is not possible as $Z^{u}$ is the only conjugate of $Z$ in $(ZD)^{u}$. This contradiction finally proves
that $Z$ is not weakly closed in $Q$ with respect to $G$.

%
%
%
%
%
\end{proof}

Because of Lemma~\ref{ZnotweakQ} we may and do assume that for some $g \in G$ we have $Y = Z^g \le Q$ with $Y \neq Z$. Set $V = ZY$ and assume that $Y$ is chosen so that $C_{Q^g}(Z) \le S$.
Set $P= \langle Q,Q^g\rangle$ and $W= C_Q(Y)C_{Q^g}(Z)$.
\begin{lemma}\label{Pfacts1} The following hold:
\begin{enumerate}
\item $V \le Q \cap Q^g$;
\item  $Q \cap Q^g$ is normal in $P$ and is elementary abelian;
\item $[Q\cap Q^g,P]= V$;
\item $P/C_P(V) \cong \SL_2(3)$ and there are exactly $4$ conjugates of $Z$ in $V$; and
\item $|N_G(Z):H|=2$.
\end{enumerate}
\end{lemma}

\begin{proof} We have $C_Q(Y) \cong 3\times 3^{1+4}_+$ and so, as $C_Q(Y) \le H^g$, the structure of $\ov S$ {given in  Lemma~\ref{added} (i)} implies that $Z=C_{Q}(Y)'\le Q^g$. Hence (i) holds.
 {Since $[Q\cap Q^g, Q]= Z \le V$ and  $[Q\cap Q^g,Q^g] = Y \le V$, the first part of (ii) and (iii) hold.} Of course $\Phi(Q\cap Q^g)\le Z \cap Y=1$. Hence the second part of (ii) holds as well.
Since $|V|= 3^2$,  $[V,Q]= Z$ and $[V, Q^g]= Y$, we get (iv). Finally there is an element in $P$ which inverts $V$, and so we have $|N_G(Z)/H|=2$.
\end{proof}

\begin{lemma}\label{Pfacts2}
\begin{enumerate}
\item $W$ is a normal subgroup of $P$, $P/W \cong \SL_2(3)$ and $W= C_P(V)$;
\item {$Q\cap Q^g$ is a maximal abelian subgroup of $Q$,} and $W/(Q\cap Q^g)$ is elementary abelian of
order $3^4$ which,  as a $P/C_P(V)$-module, is a direct sum of two natural $\SL_2(3)$-modules;
\item  {$WQ\not \le BQ$},   $\ov W$ has order $9$ and does not act quadratically on $Q/Z$;
\item $V$ is the second centre of $S$;
\item $S = WQ$ or $\ov{S}$ is extraspecial. {Furthermore, if $|R| = 2^7$, then  $S = WQ$; and}
\item $\ov{W}$ is inverted by an involution $t\in N_P(Z)\cap N_G(S)$  which inverts $Z$.
\end{enumerate}
\end{lemma}

\begin{proof}
Since $C_Q(Y)$ normalizes $C_{Q^g}(Z)$, $W$ is a subgroup of $G$. We have that $[Q, Y,
C_{Q^g}(Z)]=[Z,C_{Q^g}(Z)]=1$ and $[Y, C_{Q^g}(Z), Q]=1$ and so $[Q, C_{Q^g}(Z), Y]= 1$ by the Three Subgroup
Lemma. Thus $[Q,C_{Q^g}(Z)] \le C_Q(Y) \le W$. Hence $[W,Q] \le W$ and similarly $[W,Q^g]\le W$. So $W$ is a
normal subgroup of $P$. Furthermore, $[C_P(V),Q]\le C_Q(Y) \le W$ and $[C_P(V),Q^g]\le C_{Q^g}(Z) \le W$ and so
$P/W$ is a central extension of $P/C_P(V)$. Let $T$ be a Sylow $2$-subgroup of $O^3(P)$. Then as $O^3(P)/W$ is
nilpotent, $Q$ normalizes and does not centralize $T$. It follows that $P= WTQ$ and then the action of $Q$ on $T$
and the fact that $T/C_T(V)\cong\Q_8$ implies that $T \cong \Q_8$ and that $P/W \cong \SL_2(3)$, as by \cite[Satz
V.25.3]{Hu} the Schur multiplier of a quaternion group is trivial. This proves  (i).

Since $WQ = C_{Q^g}(Y)Q$ and $Y \le Q$, we have $\ov W$ is elementary abelian. Furthermore, as $Q$ is
extraspecial and as $Q\cap Q^g$ is elementary abelian by Lemma~\ref{Pfacts1} (iii), $Q\cap Q^g$ has index at
least $3^3$ in $Q^g$. Because $C_{Q^g}(Y)$ has index $3$ in $Q^g$, there is an integer $a$ such that
$$3^2 \le |\ov W|=  |WQ^g/Q^g|=3^a \le 3^3.$$ Furthermore, we have that $W/(Q\cap Q^g)=C_Q(Y)C_{Q^g}(Z)/(Q\cap Q^g) $ has order $3^{2a}$ and is elementary abelian.
If $C_{W/(Q\cap Q^g)}(Q)> C_{Q}(Y)/(Q\cap Q^g)$, then $C_{W/(Q\cap Q^g)}(Q) \cap C_{Q^g}(Z)/(Q\cap Q^g) > 1$ and
is centralized by $P$. As $P$ acts transitively on the subgroups of $V$ of order $3$, we get  $$C_{W/(Q\cap
Q^g)}(Q) \cap C_{Q^g}(Z)/(Q\cap Q^g)  \le Q/(Q\cap Q^g)$$ which is absurd. Hence $C_{W/(Q\cap
Q^g)}(Q)=C_{Q}(Y)/(Q\cap Q^g)$. In particular, $C_{W/(Q\cap Q^g)}(P)=1$ and $[W,Q](Q\cap Q^g)/(Q\cap Q^g)$ has
order $3^a$. Since $Q$ acts quadratically on $W/(Q \cap Q^g)$, as a $P/W$-module, we have that $W/(Q\cap Q^g)$ is
a direct sum of $a$ natural $\SL_2(3)$-modules.

Assume that $|\ov{W}| = 3^3$. Then $WQ = BQ$ {and so $|[Q/Z,W]|= |[Q/Z,B]| \le 3^3$. Since
  $[W, Q\cap Q^g] \ge [Q\cap Q^g, C_{Q^g}(Y)]= Y$ and $|[W/(Q\cap Q^g),Q]|= 3^a=27$, }we infer
that $3^3=|[Q/Z,W]| \ge 3^4$ which is a contradiction. This proves  (ii).

{Suppose that $WQ \le BQ$  (which is equivalent to $W$ acting quadratically on $Q/Z$). }Then $[Q,W]V/V \le
Z(W/V)$ and as $(Q\cap Q^g)/V \le Z(W/V)$, we infer that $C_Q(Y)/V \le Z(W/V)$ and this means that $W/V$ is
abelian. Since $W$ is generated by elements of order $3$, we then have that $W/V$ is elementary abelian. Letting
$t$ be an involution in $P$, we now have that $W_1=[W,t]$   has order $3^6$, is abelian and is normal in $P$. Now
by (ii) $W_1/V$ is a direct sum of two natural $P/W$-modules and so there are exactly four normal subgroups of
$P$ in $W_1/V$ of order $3^2$. Let $U$ be such a subgroup. Then $[U,Q\cap Q^g]\le V$. By (ii) we have $C_Q(Q \cap
Q^g) = Q \cap Q^g$ and so $[U,Q\cap Q^g] \not= 1$. As $[U,Q\cap Q^g]$ is normal in $P$ we get $[U,Q\cap Q^g]= V$.
Therefore $|[U,Q]/Z|= 3^2$.  {Now, as $WQ\le BQ$, $WQ$ normalizes $Q_1$, $Q_2$ and $Q_3$, so , as
$|[U,Q]/Z|=3^2$, $UQ$ centralizes exactly one of $Q_1/Z$, $Q_2/Z$ and $Q_3/Z$. This is true for all four
possibilities for  $U$. Hence there exists two candidates for $U$ centralizing $Q_1/Z$ (say). Thus  $\ov W$
centralizes $Q_1/Z$ and we get $[Q/Z,W]=[Q_2Q_3/Z,W]$ has order $3^2$. Since $|[Q/Z,W]|=3^3$, this is a
contradiction.
  Hence $W \not\le BQ$  and $W$ does not act quadratically on $Q/Z$.}  This proves (iii).

Since $W \not \le BQ$ and $|\ov W \cap \ov B|\neq 1$, we see that $C_{Q/Z}(W) = V/Z$ by using Lemma~\ref{added}
(iv). This then gives (iv).

{Note that, by (iv),  $S=C_S(Y)Q$ and so $WQ$ is normalized by $S$. Since, by Lemma~\ref{added} (i),  $\ov S$ is
isomorphic to a subgroup of $3\wr 3$ with $\ov B$ being the subgroup of $\ov S$ meeting the base group of the
wreath product, the possibilities for $\ov S$ now follow  as $\ov W$ is normalized by $\ov S$. In the case when
$|R| = 2^7$, we have that $|R/Z(R)| = 2^4$ and so does not admit an extraspecial group of order 27. Hence in this
case we get $\ov S= \ov W$ has order $9$. This proves  (v).}

Finally we note that the involution $t$ in a Sylow $2$-subgroup of $P$ inverts $Z$, normalizes $S$ and also
inverts $\ov W$. So (vi) holds.
\end{proof}

\begin{lemma}\label{Horder} One of the following holds:
\begin{enumerate}
\item $|R|= 2^9$, $S=WQ$ and  either $|H|=2^9\cdot 3^9$, $H=WRQ$ and  $$\ov H \approx (\Q_8 \times \Q_8 \times \Q_8).3^2$$ or $|H|=2^{10}\cdot3^9$,   $H/BRQ \cong \Sym(3)$  and  $$\ov H \approx (\Q_8 \times \Q_8 \times \Q_8).3. \Sym(3);$$
\item  $|R|= 2^9$, $\ov S$ is extraspecial and either $|H|=2^9\cdot 3^{10}$, $H=SR$
$$\ov H \approx (\Q_8 \times \Q_8 \times \Q_8).3^{1+2}_+$$
 or $|H|=2^{10}\cdot 3^{10}$, $H/BRQ \cong \Sym(3)$ and $$\ov H \approx (\Q_8 \times \Q_8 \times \Q_8).3^{1+2}_+ .2;$$ or
\item $|R|= 2^7$, $S=WQ$ and  either $|H|=2^7\cdot 3^9$,  $H=QRW$ and
$$\ov H \approx 2^7.3^2$$
or $|H|=2^{8}\cdot3^9$, $H/BRQ \cong \Sym(3)$ and
$$\ov H \approx 2^7.3.\Sym(3).$$
\end{enumerate}
\end{lemma}

\begin{proof} This is a summary of things we have learnt in Lemma~\ref{Pfacts2}
combined with the fact that $\ov H$ embeds into $\Sp_2(3) \wr \Sym(3)$.
\end{proof}

{We may now fill in the details of the structure of  $N_G(Z)$ and while doing so establish some further notation
which will be used throughout the remainder of the paper.}

 By
Lemma~\ref{Pfacts2} (i), $\ov W$ does not act quadratically on $Q/Z$. Thus $W \not \le QB$. It follows that
$N_{S}(R)$ contains an element $w$ which permutes $\{Q_1, Q_2,Q_3\}$ transitively ($w$ is a wreathing element).
Furthermore, as $\ov W$ is abelian, $\ov W\cap \ov B$ contains an a cyclic subgroup which is centralized by $wQ$.
We let $x_{123}$ be the corresponding element in $N_S(R)$ (here the notation should remind the readers (and the
authors) that $x_{123}$ acts non-trivially on $Q_1/Z$, $Q_2/Z$ and $Q_3/Z$ and on $R_1/\langle r_1\rangle$,
$R_2/\langle r_2\rangle$, $R_3/\langle r_3\rangle$. Since $x_{123}$ centralizes $r_1r_2r_3$, it normalizes
$\langle q_1,q_2,q_3\rangle$ and consequently $$[x_{123}, \langle q_1,q_2,q_3\rangle ]\le \langle
q_1,q_2,q_3\rangle \cap Z=1.$$ Hence $x_{123} \in C_{S}(\langle q_1,q_2,q_3\rangle)$.

If $S>QW$, then $|\ov B|$ has order $9$ and is normalized by $w$. Thus $N_S(R)$ contains an element
$x_2x_3^{-1}$, which as with $x_{123}$  centralizes $\langle q_1,q_2,q_3,Z\rangle$. Note that at this stage it
may be that $x_{123}$ and $x_2x_3^{-1}$ do not commute. We continue our investigations under the assumption that
if $S= WQ$, then $x_2x_3^{-1}$ is the identity element and $J=J_0$.

Set $A = [Q,B]= \langle Z, q_1,q_2,q_3\rangle$, $$J = C_{QW}(A)= \langle A, x_{123} \rangle$$ and $$J_0 =
C_S(A)=\langle A,x_{123},x_2x_3^{-1}\rangle.$$

\begin{lemma}\label{J} \begin{enumerate}
\item $J= J(W)$ is the Thompson subgroup of $W$, $ (Q\cap Q^g)J/(Q\cap Q^g)$ is a non-central $P$-chief factor
and $A \neq Q\cap Q^g$;
\item if $S>QW$ then $J_0$ is elementary abelian and $\ov{J_0}= \ov B$;
\item $x_{123}$ has order $3$ and, if $S> QW$, $x_{2}x_3^{-1}$ also has order $3$ and commutes with $x_{123}$;
\item if $S=WQ$, then $J= J(S)$ and, if $S>WQ$, then $J_0= J(S)$; and
\item if $S > QW$, then $|J_0|= 3^6$ and $S=QWJ_0$.
\end{enumerate}
\end{lemma}

\begin{proof} Because $A$ has index $3$ in $J$, $J$ is abelian.
As $J$ centralizes $V$ and $J \le QW$, $J \le C_{QW}(V)= W$. {As, by Lemma~\ref{Pfacts2} (ii)}, $W/(Q\cap Q^g)$
is a direct sum of two natural $\SL_2(3)$-modules, there is a normal subgroup $W_0$  of $P$ such that $(Q\cap
Q^g) \le W_0 \le W$ and $$\ov W_0 = \ov{\langle x_{123}\rangle}\le \ov B.$$

We have $|W_0 \cap Q:Q\cap Q^g|=3$. Thus, as  $Q \cap Q^g$ is a maximal abelian subgroup of $Q$ by Lemma~\ref{Pfacts2} (ii),  $Z(W_0 \cap Q)$ has index $3$ in $Q \cap Q^g$ and contains $V$. Hence $Z(W_0 \cap Q)$ is normal in $P$ by Lemma~\ref{Pfacts1} (iii) and this means that $Z(W_0) = Z(W_0)\cap Q$.  From the definition of $A$ and of $W_0$, we have $[A,W_0]\le Z$. On the other hand, $Z(W_0) \le C_Q(W_0) \le A$. Thus $W_0$ centralizes a subgroup of $A$ of index $3$. It follows that $W_0$ induces a group of order $3$ on $A$. Hence $C_{W_0}(A) = J$ and $W_0 = (Q \cap Q^g)J$. As  $[W_0,Q \cap Q^g] = V$, $W_0$ is
not abelian and hence $J$ is a maximal abelian subgroup of $W_0$.

If $J^*\le W_0$ is abelian with $|J^*|= |J|$
and $J\neq J^*$, then $W_0= JJ^*$ and $Z(W_0)\ge J\cap J^*$. Since $Q\cap Q^g \not \le Z(W_0)$ and $W_0/(Q\cap
Q^g)$ is a $P$-chief factor, we get $W_0= Z(W_0)(Q\cap Q^g)$ which means that $W_0$ is abelian and is a
contradiction. Hence $J= J(W_0)$ is normal in $P$ and, as $J = [J,Q][J,Q^g]$ is generated by elements of order
$3$, $J$ is elementary abelian.

Since $J$ contains a $P$-chief factor, we have $C_P(J) = C_W(J)= J$. Assume that $\wt A$ is an abelian subgroup
of $QW$ with $|\wt A|\ge |J|=3^5$. If $\wt AQ\not \le BQ$, then $|C_{Q/Z}(\wt A)| \le 3^2$ by
Lemma~\ref{added} (ii). Hence $|\wt A \cap Q| \le 3^3$ which means that $\wt A Q= WQ$ and so we have $|C_{Q/Z}(\wt
A)|=3$  by Lemma~\ref{added} (iv). But then $\ov W$ has order greater than $9$, a contradiction. So $\wt A \le
W_0Q$ and $|\wt A \cap Q|= 3^4$, it follows that $\wt A \cap Q= A$ and $\wt A \le J$. Thus $J = J(WQ)$ and if $S
= QW$ we even have $J = J(S)$.
 This completes the proof of (i) and shows that $x_{123}$  has order $3$.
 Since $J$ does not centralize $Q \cap Q^g$, $A \neq Q\cap Q^g$.

%

Now we consider $J_0$ and suppose that $S > QW$. Then $S= J_0QW$. Because $A$ is normalized by $S$, $J_0$ is a normal subgroup of
$S$ and $x_2x_3^{-1}\in J_0 \setminus J$.   Set $A_1= A \cap Q\cap Q^g$. Then, as $W_0 \cap Q= A(Q\cap Q^g)$, we have $A_1$ has order $3^3$ and is centralized by $W_0J_0$.  It follows that $W_0J_0= C_S(A_1)$.
 Since $A_1$ is normalized by $P$ by Lemma~\ref{Pfacts1}(iii) and $C_{PS}(A_1) \le O_3(PS)$, we have $J_0W_0$ is normalized by $PS$ and that $J_0W_0/J$ is centralized by $O^3(P)$. As $J_0$ is normalized by $S$, we have that $J_0$ is a normal subgroup of
$PS$. Employing the fact that  $A \le Z(J_0)$, yields $J=\langle A^P\rangle \le Z(J_0)$. Hence $J_0$ is abelian. As $J$ is
elementary abelian, $\Phi(J_0)$ has order at most $3$ and as $P$ does not normalize $Z$ we have  $J_0$ is
elementary abelian. This then implies that  { $x_2x_3^{-1}$ has order $3$ and $[x_{123},x_2x_3^{-1}]=1$}. Since $|J_0|= 3^6$, we also
have that $J_0 = J(S)$ in this case.
\end{proof}

The next lemma just reiterates what  we have discovered in Lemma~\ref{J} (iii).

\begin{lemma}\label{BEA} $B= \langle x_{123},x_2x_3^{-1},z\rangle$ is elementary abelian.\qed
\end{lemma}

\begin{lemma}\label{controlfus}
The subgroup $N_G(J(S))$ controls $G$-fusion of elements in $J(S)$.
\end{lemma}
\begin{proof} This follows from lemma~\ref{ptransfer} (i) as $J(S)$ is elementary abelian.  \end{proof}

\begin{lemma}\label{fusion1} $N_G(Z)$ controls $G$-fusion of elements of order $3$ in $Q$ which are not conjugate to $z$.
In particular, $q_1$, $q_1q_2$ and $z$ represent  distinct $G$-conjugacy classes of elements of $Q$.
\end{lemma}

\begin{proof}  From Lemma~\ref{Pfacts2} (iii) and (v)  no element of $\ov S$  centralizes a subgroup of index $3$ in $Q$. Furthermore, if $Z^g \le Q$, then all the elements of $ZZ^g$ are $G$-conjugate to elements of $Z$ by Lemma~\ref{Pfacts1} (iv). Hence $N_G(Z)$ controls $G$-fusion of elements of order $3$ in $Q$ which are not conjugate to elements of $z$ by Lemma~\ref{Fusion}.

 By Lemma \ref{Pfacts2}(iv) any conjugate of $z$ in $Q$ is
in the second centre of some Sylow 3-subgroup of $N_G(Z)$ and so $q_1$ and $q_1q_2$ both are not conjugate to $z$
in $G$.
\end{proof}

\begin{lemma}\label{NHJ}
 We have $N_H(J) = \Omega_1(Z(R))N_H(S)$.
\end{lemma}

\begin{proof} We know by direct calculation  that $N_{\ov H}(\ov J)= \ov{\Omega_1(Z(R))}N_{\ov H}(\ov S)$  and so the result follows.

\end{proof}

Recall that, for $i=1,2,3$, $Q_i= \langle q_i,\wt q_i\rangle $ where $[q_i,\wt q_i]=z$ are specifically
defined. In the next lemma we give precise descriptions, some of which we have already seen,  of a number of the
key subgroups of $Q$.

{\begin{lemma}\label{QQg} The following hold:
\begin{enumerate}
\item $V = \langle z, q_1q_2q_3 \rangle$;
\item $C_Q(V) = \langle A,\wt q_1\wt q_2^{-1}, \wt q_1 \wt q_2\wt q_3\rangle$;
\item $A=\langle z,q_1,q_2,q_3\rangle$;
\item $A \cap Q^g= \langle V, q_1q_2^{-1}\rangle= \langle V, q_2q_3^{-1}\rangle$; and
\item $Q\cap Q^g=\langle A\cap Q^g,  \wt q_1 \wt q_2\wt q_3\rangle$.
\end{enumerate}
\end{lemma}

\begin{proof} We have that $V$ is centralized by $W$ and $\ov W = \langle wQ,x_{123}Q\rangle$, hence (i) holds and
(ii) follows from that. Part (iii) is the definition of $A$. Since $A \le C_Q(V)\le W$, $[A,W]= [A,w] \le Q^g$
and this gives (iv). Finally, since  $[Q\cap Q^g, W]= V$ and so we get (v).
\end{proof}}

\begin{lemma}\label{Zconjs1} $A$ contains exactly $13$ conjugates of $Z$ and $A\cap Q^g$ contains exactly $4$ $G$-conjugates of $Z$.
\end{lemma}

\begin{proof}
Since the images of $G$-conjugates of $Z$ contained in $Q$ are $3$-central in $N_G(Z)/Z$ by Lemma~\ref{Pfacts2}
(iv), the conjugates of $Z$ in $Q$ are $N_G(Z)$-conjugate to $\langle q_1q_2q_3\rangle$ by Lemma~\ref{fusion1}.
{Therefore, in $A = \langle z, q_1,q_2,q_3\rangle$ we have thirteen candidates  for such subgroups and they are
in the four groups $$\langle Z, q_1q_2q_3\rangle\;, \langle Z, q_1q_2^{-1}q_3\rangle\;,\langle Z,
q_1q_2q_3^{-1}\rangle\text{ and }\langle Z, q_1q_2^{-1}q_3^{-1}\rangle.$$ As all these groups are conjugate in
$\Omega_1(R)Q$, we see that $A$ contains exactly thirteen conjugates of $Z$. Now $A \cap Q^g= \langle
z,q_1q_2^{-1},q_2q_3^{-1}\rangle$ contains four conjugates of $Z$ all of which are contained in $V$.}
\end{proof}

\begin{lemma}\label{conjugates in Z} $J_0$ contains exactly $40$ subgroup which are $G$-conjugate to $Z$ and they are all
contained in $J$. In particular, $N_G(J) \ge N_G(J_0)$ and $|N_G(J)/J_0| = 2^{7+i}\cdot 3^4\cdot 5$ where $i$ is
such that $2^{i+2}= |N_G(S)/S| \le 8$.
\end{lemma}

\begin{proof} By Lemma~\ref{Zconjs1}, we have that $A= J\cap Q$ contains exactly thirteen conjugates of $Z$ and
$J \cap Q^g \cap Q= A \cap Q\cap Q^g= \langle z,q_1q_2^{-1},q_2q_3^{-1}\rangle$ contains exactly four conjugates
of $Z$. We have that both $J$ and $J \cap Q \cap Q^g$ are normal in $P$. As $J/(J \cap Q \cap Q^g)$ is a natural
$P$-module by Lemma \ref{Pfacts2}(ii), we see that $J = \cup_{x \in P}(J \cap Q)^x$ {is a union of four
conjugates of $J\cap Q$ pairwise meeting in $J\cap Q\cap Q^g$. This gives, using the inclusion exclusion
principle and Lemma~\ref{fusion1}, that there are exactly $4 \cdot 13 - 3 \cdot 4 = 40$ conjugates of $Z$ in
$J$.} In particular, $J_0 =\langle Z^g \mid Z^g \le J_0\rangle$.

Suppose that $J_0 >J$.
 Then $|R| = 2^9$ and $\ov S \cong 3^{1+2}_+$. If $N_G(J_0)$ normalizes $J$ then Lemma~\ref{controlfus} delivers the result. So we may
assume that $N_G(J)$ does not normalize $J_0$.
  Suppose that $X$ is a subgroup of $J$ of order $3$ and that $X \not \le
J_0$.  Then $\ov X\le \ov B$ and $\ov X \not= \ov J_0$ is conjugate to $\ov{\langle x_{2}x_3^{-1}\rangle}$ and so
we have that $C_{Q}(X)$ is conjugate to $Q_1A$ which has order $3^5$. Thus $XA$ is normalized by $Q$, $|X^Q|=
3^2$ and, {as $|(XQ)^{S}|=3$,} $|X^S| = 27$.

Hence, taking $X$ to be a conjugate of $Z$, yields that there are $40 + 27i$ conjugates of $Z$ contained in $J_0$
where $1 \le i \le 9$. If there is some non-trivial element of $A$ which has all its $G$-conjugates contained in
some proper subgroup of $J$, then we have that this subgroup is normal in $N_G(J_0) \ge S$ and so contains $Z$.
But then $Z$ is trapped in this subgroup, a contradiction. By Lemma \ref{fusion1} there are at least two
$G$-conjugacy classes  of cyclic subgroups different from $Z$ in $A$ and so
 there are at
least 54 cyclic subgroups of $J_0$ not in  $J$, which are not $G$-conjugate to $Z$. It follows that $i \le 7$.
Now the only non-zero $i$ which has $40 + 27i$ dividing $|\GL_6(3)|$ is $i=3$. This means that there are 121
conjugates of $Z$ in $J_0$ and that $N_G(J_0)$ contains a cyclic group $D$ of order $121$. Let $J_1 \le J$ have
order $3^5$ be normalized by $D$. Then $D$ acts transitively on the cyclic subgroups of $J_1$  and consequently
$J_1\cap Q = J_1 \cap A$ which has order $27$ has only one $G$-class of cyclic subgroups. As $Z \not\leq J_1 \cap
A$, we get that $(J_1 \cap A)Z = A$. Now all elements of  $A$ not in $Z$ are conjugate, which contradicts
Lemma~\ref{Zconjs1}. Now we have that all the $G$-conjugates of $Z$ in $J_0$ are contained in $J$. Thus $N_G(J_0)
\le N_G(J)$.

\end{proof}

\begin{lemma}\label{q1conjs}  There are $36$ conjugates of $\langle q_1\rangle$ in $J$. In particular, $\langle q_1\rangle$ is centralized by an element of order $5$ in $N_G(J_0)$
\end{lemma}
\begin{proof} In $J \cap Q$, there are nine $N_H(J)$-conjugates of $\langle q_1\rangle$
(which are already conjugate in $QW$) and in $Q\cap Q^g\cap J$ there are none by Lemmas~\ref{fusion1} and
\ref{QQg} (iv). Again as $J$ is the union of the four $P$-conjugates of $J \cap Q$, we
 have $4\cdot 9$ conjugates of $\langle q_1\rangle$ in $J$.
 Since, by Lemma~\ref{conjugates in Z}, $|N_G(J_0)|$ is divisible by $5$,
 we have that some element of order $5$ in $N_G(J_0)$ centralizes $\langle q_1\rangle$.
\end{proof}

\begin{lemma}\label{N_G(J)}
$N_G(J)/J_0 \cong \Omega_5(3).2$ or $\Omega_5(3).2 \times 2$. In particular, $r_1$ centralizes an element of order $5$ in $N_G(J)$.
\end{lemma}

\begin{proof}{ Let $M=  N_G(J)$, $\mathcal P = Z^M$  and $\mathcal L = V^M$. We call the elements of $\mathcal P$
points and those in $\mathcal L$ lines.  For $X\in \mathcal P$ and $Y \in \mathcal L$, declare $X$ and $Y$ to be
incident if and only if $X \le Y$. We claim the this makes $(\mathcal P, \mathcal L)$ into a generalized
quadrangle with parameters $(3,3)$.

For $X =Z^m \in \mathcal P$, $m \in M$, we set $Q_x =O_3(C_G(x))=Q^m$.

By Lemma~\ref{Pfacts1} (iv), we have $4$ points on each line. Suppose that $Z \le V^m\in \mathcal L$. Then either
$Z^m = Z$ or $Z^m \neq Z$ and $Z \le Q^m$. In the first case $m \in H\cap M$ and $V^m \le J\cap Q_Z$ and, in the
second case, we have $Z^m \le Q$ by Lemma~\ref{Pfacts1} (i) and so    $V^m \le Q_Z$ again. Thus, if $X \in
\mathcal P$ is incident to a line $L\in \mathcal L$, then $L \le J\cap Q_X$.

By Lemma~\ref{Zconjs1}   there are twelve $M$-conjugates of $Z$ in $(J \cap Q) \setminus Z$ and each of them
forms a line with $Z$. Thus $Z$ is contained in exactly $4$ lines and, furthermore, any two lines containing $Z$
meet in exactly $Z$ and any two points determine exactly one line.

Now suppose that $L\in \mathcal L$ is a line which is not incident to $X \in \mathcal P$. Then, as $|J:J\cap
Q_X|= 3$, we have $L \cap (J\cap Q_X)$ is a point and this is the unique point of $L$ which is collinear to $X$.
It follows that $(\mathcal P,\mathcal L)$ is a generalized quadrangle with parameters $(3,3)$.} By \cite{Payne1}
there is up to duality a unique such quadrangle. Hence we have that $N_G(J)/J_0$ induces a subgroup of
$\Omega_5(3).2$ on the quadrangle. Using Lemma \ref{conjugates in Z}, we see that the full group is induced. As
there might be some element which inverts $J$ and so acts trivially on $(\mathcal P,\mathcal L)$, we get the two
possibilities as stated.

Finally, as $r_1$ acts as a reflection on $J$, we see that $r_1$ centralizes an element of order 5.\end{proof}

\begin{lemma}\label{alt6} We have $F^*(C_{N_G(J)}(q_1)/J_0) \cong \Alt(6) \cong \Omega_4^-(3)$.\end{lemma}

\begin{proof}  Because $q_1$ is inverted by $r_1$ and $r_1$ acts on $J$ as a reflection, we have that  $F^*(C_{N_G(J)}(q_1)/J_0)$ is an orthogonal group in dimension $4$. Since, by Lemma~\ref{q1conjs}, $q_1$ commutes with an element of order $5$, we have $F^*(C_{N_G(J)}(q_1)/J_0)\cong \Omega_4^-(3)\cong \Alt(6)$.
\end{proof}

\section{The Fischer group $\M(22)$ and its automorphism group}

In this section we will assume that $|R| = 2^7$ and determine the isomorphism type of $G$. Set $r = r_1$ and  $K
= C_G(r)$. Recall that $R$ is a subgroup of $R_1\times R_2\times R_3 \cong \Q_8\times Q_8\times\Q_8$ and $R \ge
\langle r_1, r_2, r_3\rangle= \Omega_1(Z(R))$.

\begin{lemma}\label{OmegaR} We have that $\Omega_1(Z(R)) \leq \Phi (R)$.
\end{lemma}

\begin{proof} Assume that $\Omega_1(Z(R)) \not\leq \Phi (R)$. As $w$ acts transitively on the set $\{r_1,r_2,r_3\}$, we
may assume that $r_i \not\in \Phi(R)$ for $1\le i\le3$.  Let $U$ be a hyperplane in $\Omega_1(Z(R))$ which
contains $\Phi(R)$. Then, as $w$ normalizes $R$, we may assume that $\{r_1,r_2,r_3\} \cap U = \emptyset$. An easy
inspection of the maximal subgroups of $\Omega_1(Z(R))$ yields $U = \langle r_1r_2, r_2r_3 \rangle$. Therefore
$(R_1 \times R_2 \times R_3)/U$ is an extraspecial group of order $2^7$. We have that $R/U$ is of order $2^5$,
hence $R/U$ is not abelian. However $\Phi (R) \not\leq U$, which is a contradiction.  \end{proof}
%

{Recall from Lemma~\ref{Horder} (iii), either $H =QRW$ or $H/BRQ \cong \Sym(3)$ and in either case $S=WQ$. If
$H/BRQ \cong \Sym(3)$, then there is an element $iRQ$ of order $2$ which permutes $Q_2$ and $Q_3$ and centralizes
$r$. We let $i\in H$ be such an element where for convenience we understand that $i=1$ if $H=QRW$. Thus in any
case $H=QRW\langle i\rangle$. By Lemma~\ref{Pfacts2} (vi), $|N_G(Z):H|=2$ and $\ov W$ is
inverted by an involution $j$
 in $N_G(Z) \cap N_G(S)$. Again, we can choose $j$ to centralize $rQ \in HQ$ and consequently it can be further
 chosen to centralize $r$. Thus we have
$N_K(Z) = Q_2Q_3RC_S(r) \langle i,j \rangle$ and this group has order $3^6\cdot 2^9$.}

\begin{lemma}\label{Cr1Fischer} Suppose that $|R| = 2^7$. Then $K \cong 2\udot \U_6(2) $ or $2\udot \U_6(2).2$.
\end{lemma}

\begin{proof} We have $N_K(Z) = Q_2Q_3RC_S(r) \langle i,j \rangle$. Since $Z(C_S(r)R/\langle r \rangle)$ acts faithfully on $Q_2Q_3$ and centralizes the fours group
$\Omega_1(R)/\langle r\rangle$, we see that $N_K(Z)/\langle r\rangle$ when embedded into $\GSp_4(3)$ preserves
the decomposition of the associated symplectic space into a perpendicular sum of two non-degenerate spaces and
has  $R/\langle r \rangle \cong \Q(8) \times \Q(8)$ as a normal subgroup.   Therefore, as $Q_1Q_2\cong
F^*(N_K(Z)/\langle r \rangle)$ is extraspecial of order $3^5$, we have
  $N_K(Z)/\langle r\rangle$ is similar to a normalizer in a group
of $\U_6(2)$-type. By Lemma~\ref{fusion1}, no conjugate of $Z$ is $G$-conjugate to an element of $Q_1Q_2\setminus
Z$ and so $Z$ is weakly closed in $Q_1Q_2$ with respect to $K$. Since, by Lemma~\ref{N_G(J)}, $C_{N_G(J)}(r)$ has
an element $f$ of order $5$, we have $Z^f \le C_J(r)$ and, of course, $Z^f \neq Z$. It follows that $Z\langle
r\rangle/\langle r\rangle$ is not weakly closed in $C_S(r )\langle r\rangle/\langle r\rangle$ with respect to
$C_G(r)/\langle r\rangle$. Therefore, as $C_S(r)Q_2Q_3/Q_2Q_3$ has order $3$, Theorem~\ref{F4U6Thm} implies
that $C_G(r)/\langle r\rangle \cong \U_6(2)$ or $\U_6(2).2$. Since $R \leq C_G(r)$ and $r \in R^\prime$ by Lemma
\ref{OmegaR},  $F^*(C_G(r))$ does not split over $\langle r\rangle$. It follows that $F^*(C_G(r)) \cong
2\udot\U_6(2)$ or $2\udot \U_6(2).2$ as claimed.
\end{proof}

Let  $K_1= F^*(K) \cong 2\udot \U_6(2)$ and  fix some Sylow 2-subgroup $T$ of $K_1$. In $T/\langle r \rangle$
there is a unique elementary abelian group of order $2^9$ with normalizer of shape $2^9:\PSL_3(4)$ (the
stabilizer of a totally isotropic subspace of dimension $3$). Let $E$ be the preimage of this subgroup. Then
$\PSL_3(4)$ acts irreducibly on $E/\langle r \rangle$ and $|E| = 2^{10}$, we get that $E$ is elementary abelian
of order $2^{10}$ with $N_{K_1}(E)/E \cong \PSL_3(4)$ and $C_G(E)=C_K(E) =E$. By \cite[(23.5 .5)]{Asch}, $E$ is
an indecomposable module for $N_K(E)/E$.


\begin{lemma}\label{m22} We have that $N_G(E)/E \cong \M_{22}$ or $\Aut(\M_{22})$.
\end{lemma}

\begin{proof} As $r^H \cap R^\prime \not= \{r\}$ we have that $r^G \cap K_1 \not= \{r\}$.
As all involutions of $\U_6(2)$ are conjugate into $E$ (see \cite[(23.3)]{Asch}), we have that $r^{N_G(E)} \not=
\{r\}$. Recall that $E/\langle r \rangle$ is just the Todd module for $\L_3(4)$ and so $N_K(E)$ has orbits of
length $1$, $21$ ,$21$, $210$, $210$, $280$ and $280$ on $E$ (where some of these lengths may double as $E$ is
indecomposable) {by \cite[(22.2)]{Asch}.}

  Then, as $Z(T)\le E$ has order $4$ by \cite[Table 5.3t]{GLS3},
$N_K(Z(T)) $ has shape $2.2^{1+8}_+.\SU_4(2)$. In particular, we can choose   $t \in Z(T)$ such that $t$ is a
square in $K_1$ and  $Z(T)=\langle r,t\rangle$. Since $r$ is not a square in $K_1$ by \cite[(23.5.3)]{Asch}, we have
 $t$ is not $N_G(E)$-conjugate to $r$.  Now taking in account that
$|N_G(E)/E|$ has to divide  $|\GL_{10}(2)|$, we see that $|r^{N_G(E)}| = 2\cdot 11$, $2^9$ or $561$.  {If $|r^{N_G(E)}|=
561$, then  $|N_G(E)/E| = 2^{a}\cdot 3^3 \cdot 5 \cdot 7 \cdot 11 \cdot 17$, where $a = 6$ or $7$. As the
normalizer of a Sylow $17$-subgroup in $\GL_{10}(2)$ has order $2^4\cdot 3^2\cdot 5 \cdot 17$, Sylow's Theorem
implies that there must be  $2^4\cdot 3\cdot 5 \cdot 7 \cdot 11$ Sylow $17$-subgroups in $N_{G}(E)/E$. In
particular the Sylow $3$-subgroup $D$ of the normalizer of the subgroup of order $17$ has order $9$ and is
elementary abelian. Two of the cyclic subgroups of $D$ are fixed point free on $E$, one has centralizer of order
$4$ and the final one centralizes a subgroup of order $2^8$. As the Sylow $3$-subgroups of $N_G(E)$ have order
$3^3$, at least one of these subgroups is conjugate in to $N_K(E)$ and there we see that such groups all have
centralizer of order $2^4$ in $E$. This shows that this configuration cannot arise.}

So assume that $|r^{N_G(E)}| = 2^9$.  {Then $|N_G(E)/E|=2^{a}\cdot 3^2 \cdot 5 \cdot 7$, $a = 15$ or $16$. Since
some orbit on $E$ is of odd length, we must have an orbit of length $21$, $231$ or $301$ or $511$. As we know
$|N_G(E)|$, we get an orbit of length $21$. From the action of $\L_3(4)$ on this set, we see that no element of
odd order fixes more than $3$ points. Let $T \in \Syl_2(N_G(E)/E)$. Now $\Sym(21)$ has Sylow $2$-subgroups of
order $2^{18}$ and $\Sym(8)$ has Sylow $2$-subgroups of order $2^6$. Hence, as $|T| \ge 2^{15}$, there is an
involution $j \in T$ which fixes at least $13$ points and the product of two such involutions fixes at least $5$
points. It follows that $\langle j,j^x\rangle$ is a $2$-group for all $x \in N_G(E)/E$. Hence $O_2(N_G(E)) > E$
by the Baer-Suzuki Theorem and this contradicts the fact that $N_G(E)$ acts irreducibly on $E$ and $C_G(E)=E$.

 So
we have that $|r^{N_G(E)}| = 22$. In particular we have that $N_G(E)/E$ acts triply transitive on 22 points with
point stabilizer $\L_3(4)$ or $\L_3(4):2$. Using, for example \cite{Lun},  get that $N_G(E)/E$ is isomorphic to
$\M_{22}$ or Aut($\M_{22})$, the assertion.}\end{proof}

\begin{proof}[Proof of Theorerm~\ref{Main1}] If $K= K_1$, then, as $r$ is not weakly closed in a Sylow $2$-subgroup of $G$ (its conjugate to
$r_2$ for example) we have $G \cong \M(22)$ by \cite[Theorem 31.1]{Asch}. If  $K> K_1$, then also $N_G(E)/E \cong
\Aut(\M_{22})$ and Lemma \ref{autm22} (ii) implies that $G$ has a subgroup $G_1$ of index $2$. We have $K_1 = K
\cap G_1$ and $G_1 \cong \M(22)$ by \cite[Theorem 31.1]{Asch}.
\end{proof}

\section{Some notation}

From here on we may suppose that $|R|= 2^9$. In this brief section we are going to reinforce some of our earlier notation in preparation for determining the centralizers of various elements in the coming sections.

We begin by recalling our basic notation which has already been established. We have $R_1$, $R_2$, $R_3$ are the
normal quaternion groups of $R$ and $Q_i = [Q,R_i]$ extraspecial of order 27. We have defined $Z(R_i) = \langle
r_i \rangle$ so that $Z(R) = \Omega_1(R)=\langle r_1,r_2,r_3 \rangle$. We have for $B = C_S(Z(R))$ and that $B =
\langle Z,x_{123}, x_2x_3^{-1} \rangle$, where the last element is non-trivial just  when  $WQ < S$.  By
Lemma~\ref{BEA} $B$ is elementary abelian. Further we have some $w \in N_H(R)$ with $Q_1^w = Q_2$, $Q_2^w = Q_3$
and $Q_3^w = Q_1$.

From Lemma~\ref{Horder} (ii) and (iii)  we have
 $|H| = 2^{9+a}\cdot 3^{10}$ or $2^{9+a}\cdot 3^9$ where $a= 0, 1$. When $a=1$, just as in the case when
 $|R|=2^7$,
there exists a further involution $i \in N_H(S)$.  This involution can be chosen to centralize $Z$ and normalize
$R$. Since, by Lemma~\ref{Horder}, $\ov H$ is isomorphic to a  subgroup of $\Sp_{2}(3) \wr \Sym(3)$, we see that
$i$ can be selected so that $Q_1$ is centralized by $i$, and so that $Q_2^i= Q_3$.

We take the involution $t \in N_P(Z) \cap N_G(S)$ from Lemma~\ref{Pfacts2} (vi). Since $t$ normalizes $QR$ and $Q
\le P$, we may assume that $t$ normalizes $R$.  Since $t$ inverts $\ov W$, $t$ inverts $wQ$ and so $t$ permutes
$R_1$, $R_2$ and $R_3$ as a $2$-cycle. Thus we may suppose that $t$ normalizes $R_1$ and exchanges $R_2$ and
$R_3$. In particular, $t$ centralizes $r_1$ and acts on $Q_1$ inverting $Z$. Since $W/(Q\cap Q^g)$ is inverted by
$t$, we see, using Lemma~\ref{QQg} (iv),  that $q_1(Q\cap Q^g)$ is inverted by $t$.  Similarly $\wt {q_1} W$ is
centralized by $t$. It follows that $[Q_1, t]= Z\langle q_1\rangle$ and that $t$ inverts $q_1$.

\begin{lemma}\label{calc1} With the notation just established, we have $N_{N_G(Z)}(R) = R\langle z,x_{123}, x_2x_3^{-1}, w\rangle\langle i,t\rangle$. Furthermore,
\begin{enumerate}
\item $q_1^t = q_1^{-1}$.
\item $t$ inverts $\langle z,x_{123},w\rangle$ which is abelian and $t$ centralizes $x_{2}x_3^{-1}$.
\item $w^i= w^{-1}$ and $(x_{2}x_3^{-1})^i=(x_{2}x_3^{-1})^{-1}$.
\end{enumerate}
\end{lemma}

\begin{proof}  We have already discussed (i).
By Lemma~\ref{Pfacts2}(iv), $t$ inverts $\ov W = \ov{\langle x_{123},w\rangle }$ and $t$ inverts $Z$. Thus, we
may choose notation so that
 that $t$ inverts  $\langle z,x_{123}, w\rangle$ (i) holds.
 Furthermore, we may suppose
that $t$ centralizes $x_2x_{3}^{-1}$. Now $C_X(i) =\langle Z,x_{123}\rangle$ and $[X,i] $ has order $9$. In
particular, $[X,i]\cap [X,t]$ has order $3$. We choose $w$ such that $[X,i]\cap [X,t]=\langle w \rangle$. Finally
we may suppose that $x_{2}x_3^{-1}$ is chosen so that it is inverted by $i$.
\end{proof}

\section{A signalizer}

 Recall from Lemma~\ref{Pfacts2} (vii) that there is an involution $t \in P$ which inverts both $Z$ and $\ov W$ and that further properties of $t$ are listed in Section 6. We set $$H_0 = QWR\langle t\rangle$$ and note that, as  $t$ inverts $\ov W$,  $H_0$ is a normal subgroup of $N_G(Z)$.

\begin{lemma}\label{u62} The following hold.
\begin{enumerate}
\item $F^*(C_G(q_1)) \cong 3 \times \U_6(2)$; \item $|N_G(\langle q_1\rangle) :C_G(q_1)|=2$; and
 \item $C_G(q_1)/F^*(C_G(q_1)) \cong N_G(Z)/H_0$ and is isomorphic to a subgroup of $\Sym(3)$.
\end{enumerate}
Furthermore $[r_1, E(C_G(q_1))] = 1$.
\end{lemma}

\begin{proof}   We have $O^2(C_H(q_1)) = C_Q(q_1) (R_2R_3)B$ which has shape
$(3 \times 3^{1+4}_+).( \Q_8\times \Q_8).3^k$ where $3^k= |\ov B|$ with $k=1,2$.   From Lemma~\ref{calc1} (i),
we have that $t$ inverts $q_1$ and, by definition $t$ inverts $Z$, since $r_1$ inverts $q_1$ and centralizes $Z$,
we have that $r_1 t\in N_{C_H(q_1)}(Z)$.  Thus $$C_{N_G(Z)}(q_1) = \langle q_1\rangle Q_2Q_3R_2R_3J_0\langle i,
r_1t\rangle.$$
Now we see that $O_3(C_{N_G(Z)}(q_1)/\langle q_1\rangle) = Q_2Q_3\langle q_1\rangle /\langle q_1\rangle$  is extraspecial of order $3^5$ and that $$O_2(C_{N_G(Z)}(q_1)/Q_2Q_3\langle q_1\rangle )= R_2R_3Q_2Q_3\langle q_1\rangle /Q_2Q_3\langle q_1\rangle /\langle q_1\rangle \cong \Q_8\times \Q_8.$$ Thus $C_G(q_1)/\langle q_1\rangle$  is similar to a $3$-centralizer in either $\U_6(2)$ or $\F_4(2)$ (see Definition~\ref{U6F4def}).
By Lemma~\ref{q1conjs}, $q_1$ is centralized by an element $f$ of order $5$ in $N_G(J)$. Furthermore, $f$ does not
normalize $Z$ as $5$ does not divide the order of $H$. Since $Z^f \le J$ and $f \in C_G(q_1)$, we see that $Z$ is
not weakly closed in $C_S(q_1)$ and so it follows from Theorem~\ref{F4U6Thm} that $F^*(C_H(q_1))/\langle
q_1\rangle \cong \U_6(2)$ or $\F_4(2)$  and that $C_H(q_1)/F^*(C_H(q_1)) \cong H/H_0$.  Finally, as $N_{C_G(q_1)}(J)$ involves $\Alt(6)$ by Lemma~\ref{alt6}, the subgroup structure of $\F_4(2)$ implies that $$F^*(C_H(q_1)/\langle
q_1\rangle )\cong \U_6(2).$$

Now $\langle q_1 \rangle $ is normalized
by the involution $r_1$ and $r_1$ centralizes $C_H(q_1)/\langle q_1\rangle$. Hence, by Proposition
\ref{centralizerinvsU6},  $r_1$ centralizes $C_G(q_1)/\langle q_1\rangle$. Since $C_{H}(q_1)$ splits over $q_1$,
we now have $F^*(C_G(q_1)) \cong 3 \times \U_6(2)$. This proves (i). Part (ii) follows as $r_1$ (and $t$) invert
$q_1$.

 We also easily have
$C_G(q_1)/F^*(C_G(q_1)) \cong N_G(Z)/H_0$.
\end{proof}

Let $K = E(C_G(q_1))$. Then $K \cong \U_6(2)$ by Lemma~\ref{u62}. Since
$R_2 \le C_G(q_1)$, we have $r_2 \in K$. As $r_2$ centralizes $Q_3\cong 3^{1+2}_+$ in $K$, Proposition \ref{centralizerinvsU6} yields   $$C_K(r_2) \cong 2^{1+8}_+:\U_4(2).$$ Notice
that $r_3$ is also in $K$ and therefore $q_2$ and $q_3 \in K$. From the structure of $C_S(q_1)$ we also have that
$z \in K$.

Furthermore, we have  $|J_0 \cap K|$ is elementary abelian of order $3^4$ and that $A \cap K = \langle
Z,q_2,q_3\rangle= C_A(r_1)$.  Using \cite[Theorem 4.8]{PS1}, we get that $$F=N_{K}(J\cap K) \cong 3^4:\Sym(6).$$
Furthermore \cite[Lemma 4.2]{PS1} indicates that $Z$ has exactly 10 conjugates under the action of $F$. As $A
\cap K = J \cap O_3(C_K(Z))$ we see that $(A\cap K)^F$ has order $10$ and  $F$ acts  $2$-transitively on this
set.

We also have that  $F$ commutes with $\langle q_1, r_1\rangle \le C_G(K)$ and $A \cap K = C_A(r_1)$. Let $f \in F$
be such that $C=(A\cap K) \cap ( A\cap K)^f = \langle q_2, q_3\rangle$.
Then, as $q_1$ and $q_2$ are $G$-conjugate,  we obtain $$L=C_G(C)^\infty \le C_{C_G(q_2)}(q_3)^\infty  \cong \U_4(2)$$ from
Lemma~\ref{u62}. In addition, $C$ commutes with $R_1R_1^f N_J(R_1)N_J(R_2)$ and therefore
$R_1R_1^f \le L  \cong \U_4(2)$. If $R_1  = R_1^f$, then $R_1$ centralizes $J \cap K$. However,
$C_J(R_1) \le Q$ and $J \cap K \not \le Q$. Therefore $R_1 \not= R_1^f$ and this means that $r_1$ is a $2$-central
involution of $L$. Hence $R_1R_1^f \cong 2^{1+4}_+$ and we deduce that $R_1$ and $R_1^f$ commute as $R_1R_1^f$ contains exactly two subgroups isomorphic to $\Q_8$.
As $F$ acts $2$-transitively on the set $(A\cap K)^F$, we deduce that any two $F$-conjugates of $R_1$ commute and so $$E=\langle R_1^F\rangle \cong 2^{1+20}_+$$ and this is a $2$-signalizer for $F$.

\begin{lemma}\label{signal1}  The following hold.
\begin{enumerate}
\item $E$ is extraspecial of order $2^{21}$ and plus type;
\item $C_E(Z)=R_1$;
\item $E$ is the unique maximal $2$-signalizer for $Q_2Q_3$ in $C_G(r_1)$; and
\item  $C_{G}(\langle r_1, q_1\rangle)$ normalizes $E$.
\end{enumerate}
In particular,  $K$ normalizes $E$.
\end{lemma}

\begin{proof} We have already remarked that (i) is true. Also, we know that $Q_2Q_3\le F$ and so $E$ is a $2$-signalizer for $Q_2Q_3$.
Suppose that $D$ is a $2$-signalizer for $Q_2Q_3$ in $C_G(r_1)$. Then $$D= \langle C_D(x) \mid x \in \langle
z,q_2\rangle^\# \rangle$$ and observe that $\langle z,q_2\rangle$ contains three  $Q_2$-conjugates of $\langle q_2\rangle$.  Now in $C_K(z)$ the only 2-subgroup which is normalized by $Q_2Q_3$ is $R_1$ and this
is contained in $E$. In particular, (ii) holds.  So we consider signalizers for $\langle q_2,Q_3\rangle$ in  $C_{C_G(r_1)}(q_2)$. First we
note that $R_1$ commutes with $q_2$ and so we have that $r_1 \in K_2=C_G(q_2)^\infty \cong \U_6(2)$ and, as
$Q_1Q_3 \le O_3(C_{K_2}(Z))$, we have that $Q_3 \le C_{K_2}(r_1)$ and this means that $r_1$ is a $2$-central
element of $K_2$ by Proposition \ref{centralizerinvsU6}. As an extraspecial group of order 27 in $\U_4(2)$ does
not normalize a non-trivial 2-group, we now have that the  maximal signalizer for $Q_3$ in $C_{C_G(q_2)}(r_1)$ is
$O_2(C_{K_2}(r_1)) \cong 2^{1+8}_+$. We have that $\langle Z, q_2 \rangle$ acts on $E$ and $ C_E(\langle Z,
q_2\rangle) = C_E(Z) = R_1$.  Since $$E =  \langle C_E(x) \mid x \in \langle z,q_2\rangle^\# \rangle,$$ we have
$|C_E(q_2)|= 2^9$ and $C_E(q_2)= O_2(C_{K_2}(r_1))$ . Therefore  $C_D(q_2) \le E$. It now follows that $D \le E$
as claimed in (iii).

From the construction of $E$, we have that $E$ is normalized by $F$ and (ii) implies that $N_{C_G(\langle
q_1,r_1\rangle)}(Q_2Q_3)= N_{C_G(\langle q_1,r_1\rangle)}(Z)$ also normalizes $E$. Now either using \cite{Atlas}
or \cite{PS1} we have that $C_{G}(\langle q_1,r_1\rangle)$ normalizes $E$. This is (iii). Since $K \le
C_{G}(\langle q_1,r_1\rangle)$ by Lemma~\ref{u62}, we have $K \le N_G(E)$ as well.
\end{proof}

\begin{lemma}\label{NE} $F^\ast(N_{G}(E)/E) = KE/E \cong \U_6(2)$.
\end{lemma}

\begin{proof} Note that $N_G(E) = N_{C_G(r_1)}(E)$. In $N_{C_G(r_1)}(E)/E$ we have that $N_K(Z)E/E$ is a $3$-normalizer  of type $\U_6(2)$.  Therefore, as $Z$ is
not weakly closed in $C_S(r)E/E$ with respect to $N_{C_G(r_1)}(E)/E$, we have that $F^\ast(N_{C_G(r_1)}(E)/E)= EK/E$ from
Theorem~\ref{F4U6Thm}.
\end{proof}

\begin{lemma}\label{syl2} $N_G(E)/E$ acts irreducibly on $E/\langle r_1\rangle$ and $N_G(E)$ contains a Sylow $2$-subgroup of $G$.
\end{lemma}

\begin{proof}  We know that $F^*(N_G(E)/E) \cong \U_6(2)$ and that $|E/\langle r_1\rangle|= 2^{20}$. The action of $F$ and $E$, shows that $E/\langle r_1\rangle$ is irreducible. Thus Lemma~\ref{Vfacts} implies that $E/\langle r_1\rangle$ is not a failure of factorization module for $N_G(E)/E$.  In particular, if $T \in \syl_2(N_G(E))$, we have that $Z(T)=\langle r_1\rangle$ and the Thompson Subgroup of $T/\langle r_1\rangle$ is $E/\langle r_1 \rangle$ by \cite[Lemma~26.15]{GLS2}. Thus $N_G(T) \le N_G(E)$  and so $T \in \syl_2(G)$.\end{proof}

We close this section with a technical detail that we shall need later.
\begin{lemma}\label{CKq2}  We have $C_K(q_2) \cong 3 \times \U_4(2)$. \end{lemma}

\begin{proof} Set $X=\langle q_2, Q_3, (J\cap K) R_3\rangle \approx 3 \times 3^{1+2}_+.\Q_8.3$. Then $X \le C_K(q_2)$. As $\langle q_2\rangle =[J\cap K, r_2]$, we have that $N_{K}(J\cap K)/(J\cap K) \cong \mathrm O_4^-(3)$, we get $C_{N_K(J\cap K)}(q_2) \approx 3^4:\Sym(4)$. Hence $C_K(q_2) \cong 3 \times \U_4(2)$ as is seen in \cite{Atlas}.
\end{proof}

%
%
%

%
%
%
%
%
%

\section{The centralizer of an outer involution}

In this section we continue our investigation of the situation when $|R|=2^9$, assume that $H/BRQ\cong \Sym(3)$ and show that $G$ has a subgroup of index $2$. Thus, by Lemma~\ref{Horder}, $$\ov H \approx (\Q_8 \times \Q_8 \times \Q_8).3. \Sym(3)$$
or $$\ov H \approx (\Q_8 \times \Q_8 \times \Q_8).3^{1+2}_+ .2.$$

Since $H/BRQ\cong \Sym(3)$, Lemma~\ref{H/Q struct} implies that the Sylow $2$-subgroup of $H$ is isomorphic to
the Sylow $2$-subgroup of $\Sp_{2}(3)\wr \Sym(3)$ and hence we may select an the involution $d$ which conjugates
$Q_2$ to $Q_3$  and centralizes an extraspecial ``diagonal'' subgroup of $Q_2Q_3$ and in addition  centralizes $Q_1$ and normalizes $S$.

\begin{lemma}\label{F42} We have $C_G(d)/\langle d \rangle \cong \F_4(2)$.
\end{lemma}

\begin{proof} Since $d$ centralizes $Z$, we have $C_Q(d)$ is extraspecial of order $3^{1+4}$. Furthermore, as $\ov B$
has order $3$ or $3^2$ we have $|C_{\ov B}(d)|=3$. Thus $C_{S}(d)$ has order $3^{6}$. Furthermore, $C_{R}(d)= R_1
\times C_{R_2R_3}(d)$ is a direct product of two quaternion groups. It follows that $C_{C_G(d)}(Z)$ is a
$3$-centralizer in a group of type  $\U_6(2)$ or $\F_4(2)$.  Since $d$ normalizes $S$, $d$ normalizes $Z_2(S) = V$ and, as $V= Z\langle
q_1q_2q_3\rangle$, $d$ centralizes $V$ (see Lemma~\ref{Pfacts1}). From the definition of $P$, we now have that $d$ normalizes $P$.
Since $d$
centralizes $V$, we have that $C_{P\langle d\rangle}(V) = \langle d\rangle W$. A Frattini Argument now shows that
$C_{P\langle d\rangle}(d)W= P\langle d\rangle$. Therefore $C_{P}(d)$ acts transitively on the non-trivial
elements of $V$. Hence $Z$ is not weakly closed in $C_Q(d)$. Now Theorem~\ref{F4U6Thm} implies that $C_G(d)/\langle d\rangle
\cong \F_4(2)$ or $\Aut(\F_4(2))$. Since $|C_H(d)|= 2^7\cdot 3^6 $ it transpires that $C_G(d)/\langle d\rangle \cong
\F_4(2)$ as claimed.
\end{proof}

\begin{theorem}\label{index2}
If $H/BRQ\cong \Sym(3)$, then $G$ has a subgroup $G^\ast$ of index $2$ which satisfies the hypothesis of Theorem~\ref{Main} and in addition has $|H \cap G^\ast/BRQ|=3$.
\end{theorem}

\begin{proof}
Now let
$T \in \syl_2(N_G(E))$ and $T_0= T \cap EK$. By Lemma~\ref{syl2}, $T \in \syl_2(G)$.
Assume that $G$ does not have a subgroup of index $2$. Then by
\cite[Proposition 15.15]{GLS2} we have that there is a conjugate $d^*$ of $d$ in $T_0$ such that $C_T(d^*) \in
\syl_2(C_G(d^*))$.  In particular, we must have $|C_{EK\langle d\rangle}(d^*)| = 2^{25}$. Using
Lemma~\ref{Vaction} (ii) we see that $d^*\not \in E$. Now note that $$C_{EK\langle d\rangle}(d^*)EK = EK\langle
d\rangle$$ by Lemma~\ref{centralizerinvs} and so we require $|C_{EK/\langle r_1\rangle}(d^* \langle r_1\rangle)| =
2^{23}$ or $2^{24}$ where in the latter case, we must have
$$C_{EK/\langle r_1\rangle}(d^* \langle r_1\rangle)> C_{EK}(d^*) \langle r_1\rangle/\langle r_1\rangle.$$
We now apply  Lemma~\ref{centralizerinvs}. As $d^* \in Y^\prime$ in the notation of Lemma \ref{centralizerinvs},
this shows that (iv) and (v) not apply. But then Lemma~\ref{centralizerinvs} provides no possibility for $d^*$.
\end{proof}

\section{Transferring the element of order $3$}

Because of Theorem~\ref{index2}, from here on we suppose that $H/BRQ$ has order 3.
In this section we show that if $S>QW$, then $G$ has a normal subgroup of index $3$ which satisfies the hypothesis of
Theorem~\ref{Main}.  So assume that $S>QW$. Then, by Lemma~\ref{Horder} (ii), $\ov S$ is extraspecial and
$|H|= 2^{9}\cdot  3^{10}$ with $$\ov H \approx (\Q_8 \times \Q_8 \times \Q_8).3^{1+2}_+.$$

\begin{lemma}\label{index3} Suppose that $S>QW$ and $|H| = 2^{9}\cdot 3^{10}$. Then $G$ has a
normal subgroup $G^*$ of index of index $3$ and $C_G(Z) \cap G^* = QWR\langle t\rangle$ is similar to a
$3$-centralizer on type ${}^2\E_6(2)$ and $Z$ is not weakly closed in $S\cap G^*$ with respect to $G^*$.
\end{lemma}

\begin{proof} We know that $S= QJ_0W$ and $N_G(Z)= QRWJ_0\langle t\rangle$ by Lemma~\ref{J}(v). From
Lemma~\ref{Pfacts2}(vi), $t$ inverts $\ov W$ and so, as $\ov S$ is extraspecial, $J_0Q/JQ \cong J_0/J$ is
centralized by $t$. Therefore $J_0\not \le N_G(Z)'$ and $S/J = J_0/J \times QW/J$. Since $J_0/J$ is a normal
subgroup of $N_G(J_0/J)$ we now have that $J_0 \not \le N_G(J_0)'$. As $J_0$ is abelian, we may use Lemma~\ref{ptransfer} (ii)to obtain $J_0 \not \le G'$.
Let $G^*$ be a normal subgroup of $G$ of index $3$. Then,
as $\ov W$ is inverted by $t$ and $Q= [Q,R]$, $S \cap G^* = QW$. It follows that $C_{G^*}(Z)= QWR$ and $M \cap
G^*= N_{G^*}(J) \not \le H$, in particular, $Z$ is not weakly closed in $S\cap G^*$ with respect to $G^*$. This
proves the lemma.
\end{proof}

\section{The centralizer of an involution}

Because of Lemma~\ref{index3}, we may now assume that $G$ satisfies the hypothesis of the Theorem~\ref{Main}
with  $S =QW$ and $H= QRW$. Thus we now have $$S= QW= Q\langle x_{123},w\rangle$$ where $x_{123}$ and $w$ are as introduced just before Lemma~\ref{J}.

\begin{lemma}\label{Sarah} We have $$C_G(q_2q_3^{-1})/\langle q_2q_3^{-1}\rangle \cong \Omega_8^+(2):3.$$
\end{lemma}

\begin{proof} Set $x=q_2q_3^{-1}$. Then $$C_Q(x) = \langle  q_1, \wt q_1, q_2, q_3, \wt q_2 \wt q_3^{-1} \rangle.$$
Furthermore $[x_{123},x] = 1$ and $[w,x] \not\in Z$. Hence we see that $$C_S(x) = C_Q(x)\langle x_{123}\rangle.$$
We also have  $C_R(x) = R_1$.  So we  have  $$C_{H}(x)= \langle q_1,\wt q_1, q_2,q_3, \wt q_2 \wt
q_3^{-1}, x_{123},R_1 \rangle$$ and  $C_{H}(x)/O_3(C_{C_G(Z)}(x))\cong \SL_2(3)$. Furthermore, $[C_{Q}(x),R_1] = Q_1$ has order $27$ and $C_{Q}(x)/\langle x \rangle$ is extraspecial of order $3^5$.

 By Lemma \ref{QQg}  we see that  $x \in Q\cap Q^g$  and  $[P,x] \leq V = ZZ^g$ by Lemma~\ref{Pfacts1}(iii). Since all the elements of the
coset $Vx$ are conjugate in $P$, it follows that we may assume that there is $U \le P$ with $U \cong \Q_8$ with $[U,x]=1$.
 Then
 $Z$ and $Z^g$ are conjugate by an element of $U$. It follows that $Z$ is not weakly closed in $C_Q(x)$ with respect to $C_G(x)$. Now we have $C_G(x)/\langle x \rangle \cong  \POmega_8^+(2):3$
by Astill's Theorem~\ref{Astill}.
 \end{proof}

Recall the subgroup $E= \langle R_1^F\rangle$ from Lemma~\ref{signal1} is normalized by $C_{J}(r_1) = J\cap K$ and that $F= N_K(J\cap K) \approx 3^4:\mathrm O_4^-(3)$. Since $r_1$ centralizes $q_2q_3^{-1}$, we have that $q_2q_3^{-1} \in J\cap K$. Furthermore, we note that $F$ has exactly $3$-orbits on the subgroups of order $3$ in $J \cap K$ representatives being $Z$, $\langle q_2\rangle$ and $\langle q_2q_3^{-1}\rangle$ and that these subgroups are in different $G$-conjugacy classes by Lemma~\ref{fusion1}. The next goal is to show that $N_G(E)$ is strongly $3$-embedded in $C_G(r_1)$. The next lemma facilitates this aim.
\begin{lemma}\label{q23sigs}
The following hold:
\begin{enumerate}
\item $C_E(q_2q_3^{-1}) \cong 2^{1+8}_+$;
\item $r_1$ is a $2$-central involution in  $E(C_G(q_2q_3^{-1}))$;
\item $C_G(r_1) \cap C_G(\langle q_2q_{3}^{-1}\rangle) \le N_G(E)$;
\item $O_2(C_{E(C_G(q_2q_3^{-1}))}(r_1)) =C_E(q_2q_3^{-1})$; and
\item  $r_1^{C_G(q_2q_3^{-1})} \cap E \not=\{r_1\}$.
\end{enumerate}
\end{lemma}

\begin{proof}
Let $D= E(C_G(q_2q_3^{-1}))$. Then $D \cong \Omega_8^+(2)$ by Lemma~\ref{Sarah} as the Schur multiplier of $\Omega_8^+(2)$ is a $2$-group.

We have seen that $R_1$ centralizes $q_2q_{3}^{-1}$ and so $r_1 \in D $. As
$\langle z, q_2q_3^{-1} \rangle \le J\cap K$ acts on $E$ and $C_E(Z) =R_1\cong \Q_8$ by Lemma~\ref{signal1} (ii), by decomposing $E$ under the action of
$\langle z, q_2q_3^{-1} \rangle$
we see that $$C_E(q_2q_3^{-1}) \cong
2^{1+8}_+.$$
Hence (i) holds. Additionally, we have  $S \cap K = C_{S}(\langle r_1,q_1\rangle) = Q_2Q_3\langle q_1, x_{123}\rangle$ and therefore
 $$C_{S \cap EK}(q_2q_3^{-1}) = \langle q_2,q_3, \wt q_2 \wt q_3^{-1},x_{123}\rangle$$  has order $3^4$.
 Using this and  \cite{Atlas} we infer  that $r_1$ is a
$2$-central element of $E(C_G(q_2q_{3}^{-1}))$ which is (ii).

Since  $r_1$ is 2-central in $D$,
$$C_{C_G(q_2q_{3}^{-1})}(r_1) \approx  ((2^{1+8}_+.(\Sym(3)\times \Sym(3)\times \Sym(3)).3) \times 3$$ with
$O_2(C_{C_G(q_2q_{3}^{-1})}(r_1) ) = C_E(q_2q_3{-1})$ normalized by $C_J(r_1)$. It follows  that
 $$C_{C_G(\langle q_2q_{3}^{-1} \rangle)}(r_1) = O_2(C_{C_G(q_2q_{3}^{-1})}(r_1) )N_{C_G(r_1)}(C_J(r_1)) \le N_G(E).$$ Thus (iii) and (iv) hold.

This proves the main part of the lemma and the remaining part
follows as $r_1$ is not weakly closed in $C_E(r_1)$ in $D$.
\end{proof}

\begin{lemma}\label{3-embedded} If $N_G(E) <C_G(r_1)$, then $N_G(E)=KE$ is strongly $3$-embedded in $C_G(r_1)$.
\end{lemma}

\begin{proof} Let $d \in N_G(E)$ be a $3$-element. Then $d$ is conjugate in $N_G(E)$ to an element of $C_J(r_1)$ by Lemma \ref{U62J}.
We have $N_{C_G(r_1)}(S\cap KE) = N_{C_G(r_1)}(Z)$ and so to prove the lemma it suffices to show that $$C_{C_G(r_1)}(\langle d \rangle) \le N_G(E)$$ for all $d \in C_J(r_1)^\#$ by \cite[Proposition 17.11]{GLS2}.
By Lemma \ref{q23sigs} (iii) we have that $$C_{C_G(r_1)}(\langle q_2q_3^{-1} \rangle ) \le N_G(E).$$ By Lemma \ref{signal1} we have that
$$C_{C_G(r_1)}(Z) \leq N_G(E).$$

Further we have  that $C_{N_G(E)}(q_2)E/E = C_K(q_2)E/E \cong 3 \times
\U_4(2)$ from Lemma~\ref{CKq2}. Using Lemma \ref{u62} this shows that also $$C_{C_G(r_1)}(\langle q_2 \rangle) \le N_G(E).$$

By Lemma
\ref{fusion1} these subgroups $\langle q_2 \rangle$, $\langle q_2q_3^{-1} \rangle$ and $Z$ are in different conjugacy classes of $G$ and as $N_K(J\cap K)$ has three orbits on the non-trivial cyclic subgroups of $J\cap K$ we have accounted for all conjugacy classes of three elements in $N_G(E)$ and consequently $N_K(E)$ is strongly $3$-embedded in $C_G(r_1)$.
\end{proof}

\begin{theorem}\label{NGECGr} $C_G(r) = N_G(E)=KE\approx 2^{1+20}_+:\U_6(2)$.
\end{theorem}

\begin{proof}This now follows from Lemma~\ref{3-embedded} and Theorem~\ref{Not3embedded}.
\end{proof}

\section{The identification of $G$}

For the section  we set $r = r_1$, $L = C_G(r)$ and $K= E(C_G(q_1))$. From  Theorem~\ref{NGECGr} we
have $L= N_G(E)$  and from Lemma~\ref{u62} and Lemma~\ref{NE} we have $K \cong \U_6(2)$ with $L=KE \approx 2^{1+20}_+.\U_6(2)$. In particular,  $E$ is extraspecial of order $2^{21}$.

\begin{lemma}\label{controlfusionE} Suppose that $r^g \in E \setminus \langle r \rangle$ for some $g\in G$.
Define $F = \langle C_E(r^g), C_{E^g}(r) \rangle$ and  $X = \langle E, E^g \rangle$.
 Then
\begin{itemize}
\item[(i)]  $E \cap E^g$ is elementary abelian of order $2^{11}$ and is a maximal elementary abelian subgroup of $E$.
\item[(ii)] $C_{E^g}(r) \leq L$ and $C_{E^g}(r)E/E$ is elementary abelian of order $2^9$.
\item[(iii)] $C_{L}(r^g)E/E \cong 2^9.\L_3(4)$ and $O_2(C_{L}(r^g)E)=(E^g\cap L)E$.
\item[(iv)] $F$ is normal in $X$, $X/F \cong \Sym(3)$ and
$[X,E \cap E^g] = \langle r,r^g \rangle$.
\item[(v)] If $h \in G$ and $r^h \in E \setminus \langle r \rangle$, then there is some $k \in EK$ such that $r^{hk} = r^g$.
\end{itemize}
\end{lemma}

\begin{proof} Since $E$ is extraspecial of order $2^{1+20}$, $C_E(r^g)$ is a direct product of $\langle r^g \rangle$ with an extraspecial group of order
$2^{1+18}$.
As $|L^g/E^g|$ is not divisible by $2^{19}$, there is no such extraspecial group in $L^g/E^g$ and therefore  $r \in E^g$.\\

Because $\Phi (E \cap E^g) \leq \langle r \rangle \cap \langle r^g \rangle = 1$, $E\cap E^g$ is elementary
abelian. Hence, as $E$ is extraspecial, we have  $|E \cap E^g| \leq 2^{11}$. In particular, as $|C_{E^g}(r)| =
2^{20}$, we have that $C_{E^g}(r)E/E$ is an elementary abelian group of order at least $2^9$. Since the $2$-rank
of $L/E$ is $9$, we deduce that  $|C_{E^g}(r)E/E|= 2^9$ and $|E\cap E^g|= 2^{11}$. Furthermore $(E^g \cap L)E/E$
is uniquely determined. This completes the proof of parts (i) and (ii).

By Lemma \ref{Vfacts}, we have $|C_{E/\langle r \rangle}(C_{E^g}(r))| = 2$ and therefore $$C_{E/\langle r
\rangle}(C_{E^g}(r)) = \langle r, r^g \rangle/\langle r \rangle.$$ Hence we have that $C_{L/\langle r
\rangle}(\langle r, r^g \rangle/\langle r \rangle) = N_L(C_{E^g}(r))E$. This proves  (iii).

As $C_E(r^g)$ and $ C_{E^g}(r)$ normalize each other,
 $F$ is a 2-group and $$[E, C_{E^g}(r)] \leq C_E(r^g) \text{ and }[E^g, C_{E}(r^g)] \leq C_{E^g}(r)$$ which means that
 $F$ is normal in $X$. In addition, $[E, E \cap E^g] \leq \langle r \rangle$ and $[E^g, E \cap E^g] \leq \langle r^g \rangle$. So the
group $(E \cap E^g )/\langle r, r^g \rangle$ is centralized by $X$. Suppose that $f \in C_X(\langle r, r^g
\rangle)$ has odd order.  Then $f$ is in $L$ and centralizes $E \cap E^g$.  As $E \cap E^g$ is a maximal
elementary abelian subgroup of $E$ we now have that $E$ is centralized by $f$ and this contradicts
Lemma~\ref{NE}. Thus $C_X(\langle r, r^g \rangle)$ is a $2$-group. Modulo $F$ the group $X$ is generated by two
conjugate involutions, $X/F$ is dihedral. This shows that $X/F \cong \Sym(3)$, and proves (iv).

Suppose that $r^h \in E \setminus \langle r \rangle$ for some $h \in G$. Then by (iii) $r^h\langle r \rangle$ is centralized by a
maximal parabolic subgroup of $L/E$ of shape $2^9.\L_3(4)$. But this group has a 1-dimensional centralizer in
$E/\langle r \rangle$ and so $r^h$ is conjugate to $r^g$ in $L$ which proves (v).
\end{proof}

We now fix some Sylow 2-subgroup $T$ of $L$. From Lemma \ref{q23sigs} we have that $$r^{C_G(q_2q_3^{-1})}
\cap E \not = \{r\}.$$ Thus there $g \in G$  with $s =r^g \neq r$ and $s\in E$. By Lemma~\ref{controlfusionE} we
may assume that $Z_2(T) =\langle r,s\rangle$. We set $X= \langle E, E^g\rangle$, $$B= N_L(T)$$ and
$$P_1=BX.$$

For $2 \le j \le 4$, we let  $P_j \ge B$ be such that $P_j/E$ is a minimal parabolic subgroups in $L/E$
containing $B/E$ and $L= \langle P_2,P_3,P_4\rangle.$ Set $I=\{1,2,3,4\}$ and for $\mathcal J \subseteq I$ define
$P_{\mathcal J}=\langle P_j \mid j \in \mathcal J\rangle$ and $M=P_I$.

We further choose notation such that

\begin{eqnarray*}P_{34}/O_2( P_{34}  ) &\cong& \L_3(4)\\
 P_{23} /O_2(P_{23}) &\cong& \U_4(2)\text{ and }\\P_{24}/O_2(P_{24}) &\cong& \SL_2(2)\times \SL_2(4).
\end{eqnarray*} Let ${\mathcal{C}} = (M/B, (M/P_k), k \in I)$ be the corresponding chamber system. Thus
$\mathcal C$ is an edge coloured graph  with colours from $I=\{1,2,3,4\}$ and vertex set the right cosets $M/B$.
Furthermore, two cosets $Bg_1$ and $Bg_2$ form a $k$-coloured edge if and only if $Bg_2g_1^{-1} \subseteq P_k$.
Obviously $M$ acts on $\mathcal C$ by multiplying cosets on the right and this action preserves the colours. For
$\mathcal J \subseteq I$, set $M_{\mathcal J} = \langle P_k \mid k \in \mathcal J\rangle$ and $$\mathcal
C_{\mathcal J} = (M_{\mathcal J}/B, (M_{\mathcal J}/P_k), k \in\mathcal  J) \subseteq \mathcal C.$$ Then
$\mathcal C_{\mathcal J}$ is the $\mathcal J$-coloured connected component of $\mathcal C$ containing the vertex
$B$.

\begin{lemma}\label{P1B} The following hold.
\begin{enumerate}
\item $|P_1:B|=3$.
\item $\C_{1,3}=$ and $\C_{1,4}$ are generalized digons.\end{enumerate}
\end{lemma}

\begin{proof} By Lemma~\ref{controlfusionE} (iii), $P_{34}$ normalizes $Z_2(T)$.  Hence  $P_{34}$ acts on the set $\{E^h~|~ r^h \in Z_2(T)\}$  and consequently $P_{34}$ normalizes $X=\langle E, E^g
\rangle$. In particular, we have $P_1= BX$ and, as $X/O_2(X) \cong \Sym(3)$, (i) holds.
Now note  that  $$P_1P_3 = XBP_3= XP_3=P_3X= P_3BX= P_3P_1.$$
In particular,  the cosets of $B$ in  $\mathcal C_{1,3}$  correspond to the edges in a generalized digon with one part having valency $3$ and the other $5$. The same is true for $\mathcal C_{1,4}$ and so (ii) holds.
\end{proof}

Because  of Lemma~\ref{P1B},  have that $\mathcal C_1$ and $\mathcal C_2$ have three chambers and
$\mathcal C_3$ and $\mathcal C_4$ each have $5$-chambers. Furthermore, from the choice of notation we also have
that $\C_{3,4}$ is the projective plane $\mathrm{PG}(2,4)$ and that $\C_{2,3}$ is the generalised polygon
associated with $\SU_4(2)$. Furthermore, we have that $\C_{2,3,4}$ is the $\U_6(2)$ polar space.

\begin{lemma}\label{buildingup} We have $P_{12}/O_2(P_{12})
\cong \SL_3(2) \times 3$ and $P_{124} = P_{12}P_4$. In particular, $\C_{12}$ is the projective plane $\mathrm{PG}(2,2)$. \end{lemma}

\begin{proof}
We have that $C_{E/\langle r \rangle}(O_2(P_2))$  is 2-dimensional by Smith's Lemma \cite{Smith} and additionally  $P_2/C_{P_2}(C_{E/\langle r \rangle}(O_2(P_2)))\cong \SL_2(2)$. It follows that  $$C_{E/\langle r \rangle}(O_2(P_2))=Z_3(T)/\langle r \rangle.$$ Hence $P_2$ acts on $Z_3(T)$ and $O^3(P_2)$ induces $\Sym(4)$ on $Z_3(T)$
with the normal fours group inducing all transvections to $\langle r \rangle$. As $(E \cap E^g)/Z_2(T)$ is non-trivial  and  normal
in $T$, we have that $Z_3(T) \leq E \cap E^g$. Thus Lemma \ref{controlfusionE}(iv)  yields that $P_1$
normalizes and induces $\Sym(4)$ on $Z_3(T)$  where now the normal fours group induces all transvections to
$Z_2(T)$. Hence $\langle O^3(P_1), O^3(P_2) \rangle$ induces $\SL_3(2)$ on $Z_3(T)$. Furthermore, we have that
$P_{12}=\langle O^3(P_1), O^3(P_2) \rangle C_G(Z_3(T))$.

We now see that $$X=\langle O^3(P_1), O^3(P_2) \rangle = \langle E^h ~|~ r^h \in Z_3(T) \rangle.$$
Since, by Lemma~\ref{P1B} (ii) and choice of notation,  $X$ is normalized by $P_4$ and $\SL_2(4)$ is not isomorphic to a section of $\SL_3(2)$ we infer that
$O^2(P_4) \leq C_L(Z_3(T))$ and normalizes $\langle P_1, O^3(P_2) \rangle$. This shows that $C_{\langle P_1, O^3(P_2)
\rangle}(Z_3(T)) = O_2(\langle P_1, O^3(P_2) \rangle)$ as well as $P_{124}= P_{14}P_4$. Recall that $P_2 =
O^3(P_2)N_G(T)$ and $P_1 = O^3(P_1)N_G(T)$. So $P_{12} = \langle O^3(P_1), O^3(P_2) \rangle N_G(T)$ and this
completes the proof.
\end{proof}

\begin{lemma}\label{ominus} We have that $  P_{123} /O_2(\langle P_{123} \rangle) \cong
\Omega^-_8(2)$. \end{lemma}

\begin{proof}  Let $U_{23} $ be the preimage in $E$ of $C_{E}(O_2( P_{23}))$. Then, by Lemma~\ref{Vaction},
$U_{23}= [E,Et_1]$ where $Et_1$ is centralized by $P_{23}/E$. In particular, we have that $U_{23}/Z(E)$ is  an
orthogonal module for $P_{23}/O_2(P_3) \cong \U_4(2)$ and, furthermore, $U_{23}/Z(E)$ is totally singular which
means that $U_{23}$ is elementary abelian.  Since $U_{23}$ is normal in $T$,  $Z_2(T) \le U_{23}\le E\cap L^g$
which is the unique $T$-invariant subgroup of $E$ of index $2$. Now $P_3/O_2(P_3) \cong \SL_2(4) \cong
\Omega_4^-(2)$ and
$$(E^g \cap L)O_{2}(P_{23})/O_2(P_{23}) = O_2(P_3)/O_2(P_{23}).$$ As $P_3$ normalizes a hyperplane in $U_{23}/Z(E)$, we have
$[U_{23},E^g\cap L]$ has order $2^6$ and $[U_{23},E^g\cap L]$. In particular, $U_{23} \not \le E^g$ and, in fact,
$|U_{23}E^g/E^g|=2$ and is centralized by $O_{2}(P_1)E^g/E^g \in \syl_2(L^g/E^g)$. Thus $$[U_{23},E^g]=
U_{23}^g\text{ and } [U_{23}^g,E] =U_{23}.$$ Set $U_4 = U_{23}U_{23}^g$. Then, as $[U_{23},U_{23}^g]\le Z(E) \cap
Z(E^g)=1$, we have $U_4$ is elementary abelian. Furthermore, $[U_4,E^g] = U_{23}^g \le U_4$ and $[U_4,E]\le
U_{23}\le U_4$ and consequently $U_4$ is normalized by $X$. Since $X$ normalizes $P_3$ by Lemma~\ref{buildingup}
(i), we now have $\langle X, P_3\rangle =P_1P_3$ normalizes $U_4$. Note that $U_4E= U_{23}^gE= E\langle t_1\rangle$ and so
$C_{E}(U_4)$ has order $2^{15}$ by Lemma~\ref{Vaction}. Because $U_4$ is elementary abelian, we have $ U_4 \le
C_E(U_4)U_4$ and, as a $P_{23}/O_2(P_{23})$-module, $C_E(U_4)U_4/U_{23}$ has a  natural $8$-dimensional
composition factor and a trivial factor. Since $U_4/U_{23}$ is stabilized by $P_3$ and the composition factors of
$P_3$ on $C_E(U_4)/U_{23}$ are both non-trivial, we find that $U_4$ is normalized by $P_{123}$.

 Let $$\mathcal P=
\langle r \rangle^{P_{123}}\text { and }\mathcal L = \langle r, s \rangle ^{P_{123}}$$ and define incidence
between elements $x \in \mathcal P$ and $y \in \mathcal L$ if and only if $x \le y$. Of course all the points and
lines are contained in $U_4$. We claim that $(\mathcal P, \mathcal L)$ is a polar space. Because of the
transitivity of $P_{123}$ on $\mathcal P$, we only need to examine the relationship between $\langle r \rangle$
and an arbitrary member of $\mathcal L$. So let $l \in \mathcal L$. Then every involution of $l$ is $G$-conjugate
to $r$. Hence if $r^* \in l \cap E \;(= l\cap U_{23})$, then, by Lemma~\ref{controlfusionE} (v), $r^*$ is
$L$-conjugate to $r^g$. In particular, we have that $r^*$ is a vector of type $v_1$ in the notation of
Lemma~\ref{Vaction}.  Since $P_{23}$ has $3$-orbits on its $6$-dimensional module and since $U_{23}/\langle
r\rangle$ contains representatives of the three classes of singular vectors in $E/\langle r\rangle$, we infer
that $r^*$ is $P_{123}$-conjugate to an element of $\langle r,r^g\rangle$ .Thus $\langle r, r^*\rangle \in
\mathcal L$. Since $|U_4:U_{23}|=2$, we have that $\langle r \rangle$ is incident to at least one point of $l$.
Assume that $\langle r \rangle$ is incident to at least two points, $p_1, p_2$  of $l$. Then $\langle r,
p_1\rangle \le E$ and $\langle r, p_2 \rangle\le E$. Hence $l \le E$. But then $r$ is incident to every point on
$l$. Thus we have shown that $(\mathcal P, \mathcal L)$ is a polar space. Since $Z_3(T) \le U_{23}$, we have that
$(\mathcal P, \mathcal L)$ has rank either $3$ or $4$.   As the $ P_{123}$ induces $\Omega^-_6(2)$ on the lines
through $\langle r \rangle$, we get with \cite[Theorem on page 176]{Tits} that $(\mathcal P, \mathcal L)$ is the
polar space associated to $\Omega^-_8(2)$, the assertion.
\end{proof}

Combining Lemmas~\ref{P1B} and \ref{buildingup} we now have that $\mathcal C$ is a chamber system of type $\F_4$
with local parameters in which the panels of type $1$ and $2$ have three chambers and the panels of type $3$ and
$4$ have five chambers.

\begin{proposition}\label{2E62} We have  $\mathcal C$ is a building of type $\F_4$
with automorphism group $\Aut({}^2\E_6(2))$. In particular, $M \cong {}^2\E_6(2)$. \end{proposition}

\begin{proof}  The chamber systems $\mathcal C_{1,2}$ , $\mathcal C_{3,4}$ are projective planes with  parameters $3,3$ and $5,5$ and $\mathcal C_{2,3}$ is a generalized quadrangle  with parameters $3,5$. The remaining  $\mathcal  C_{\mathcal J}$ with $|J| = 2$ are all complete bipartite graph. Thus, using the language of Tits in \cite{local},  $\mathcal C$ is a chamber system of type $\F_4$.  Now suppose that $J$ of $\{1,2,3,4\}$ has cardinality three.
 Then $\mathcal C_{1,2,3}$ is the  $\mathrm O^-_8(2)$-building by Lemma~\ref{ominus} and, as $L/E \cong \U_6(2)$, we
 have $\mathcal C_{2,3,4}$ is a building of type $\U_6(2)$. Finally, Lemma~\ref{buildingup} implies that
$\mathcal C_{1,3,4}$ and $\mathcal C_{1,2,4}$ are both buildings. Since each rank $3$-residue is a building,  if
$\pi : {\mathcal{C}}^\prime \longrightarrow {\mathcal{C}}$  is the universal 2-covering of $\mathcal C$, then
$\mathcal{C}^\prime$ is a building of type $\F_4$ by \cite[Corollary 3]{local}. By \cite[Proof of Theorem 10.2
on page 214]{Tits} this building is uniquely determined by the two residues of rank three with connected
 diagram (i.e. $\U_6(2)$, $\Omega^-_8(2)$) and so  $F^*(\Aut(\mathcal{ C}^\prime))\cong {}^2\E_6(2)$.  Now we have that there is a subgroup $U$ of
 $\Aut(\mathcal {C}^\prime)$ such that $U$ contains $L$ and $U/D \cong M$ for a suitable normal subgroup
 $D$ of $U$. As $L = L^\prime$, we have that $L \leq F^*(\Aut(\mathcal {C}^\prime))$ and so $L$ is a maximal parabolic of $F^*(\Aut(\mathcal {C}^\prime))$. As
 $U \cap F^*(\Aut(\mathcal {C}^\prime)) > L$, we get
 $F^*(\Aut(\mathcal {C}^\prime)) \leq U$. As $F^*(\Aut(\mathcal {C}^\prime))$ is simple this implies that $U = M$ and
 therefore
$M\cong {}^2\E_6(2)$.  \end{proof}

\begin{theorem}\label{G=M} The group $G$ is isomorphic to ${}^2\E_6(2)$.
\end{theorem}

\begin{proof} By \cite{AschSe} we have that $M$ has exactly three conjugacy classes of involutions. In
$E \setminus \langle r \rangle$ we also have three classes $C_M(r)$-classes by Lemma \ref{Vaction}. Using
Lemmas~\ref{controlfusionE} (iv) and (v) and the fact that $E/\langle r\rangle$ does not admit transvections from
$L$, we may apply Lemma~\ref{Fusion} to see that $x^G \cap E = x^L$ for all $x \in E\setminus \{z\}$. In
particular, the three conjugacy classes of involutions in $M$ all have representatives in $E$. Further, if $x \in
G$ with $r^x \in M$, then there is $h \in M$ such that $r^{xh} \in E$. But now by Lemma \ref{controlfusionE} we
may assume that $r^{xh} = r$. Then $xh \in L \leq M$ and so $x \in M$. Hence $M$ controls fusion of 2-central
elements in $M$.

 If $Y$ is a normal subgroup of $G$, then, as $M$ contains the normalizer of a Sylow $3$-subgroup of $G$ and is
simple, we either have $M\le Y$ which means that $Y= G$ or $Y$ is a $3'$-group. Suppose the latter.  Since $r_1$
is in $M$ and is non-central, we   have $C_Y(r_1) \neq 1$. But then $C_Y(r_1) \le M$ a contradiction. Thus $Y=1$
and $G$ is a simple group. As $C_G(r_1) < M$ and $r_1^G \cap M= r_1^M$ we get with Lemma \ref{Holt}
 that   $G$ is isomorphic to one of
the following groups $\PSL_2(2^n)$, $\mathrm{PSU}_3(2^n)$, ${}^2\B_2(2^n)$ ($n\ge 3$ and odd) or $\Alt(\Omega)$.
In the first three classes of groups the point stabiliser in question is soluble and in the latter case it is
$\Alt(n-1)$. Since $M$ is neither soluble nor isomorphic to $\Alt(\Omega\setminus\{M\})$, we have a
contradiction. Hence $M=G$ and the proof of Theorem~\ref{G=M} is complete. \end{proof}

\section{The proof of Theorem~\ref{Main}}

Here we assemble  the mosaic which proves Theorem~\ref{Main}. Thus here we have $C_G(Z)$ is a centralizer of type ${}^2\E_6(2)$ and so $|R|= 2^9$. Lemma~\ref{Horder} (i) and (ii) gives the possibilities for the structure of $\ov H= H/Q$. If $|H|_2 = 2^{10}$, then Theorem~\ref{index2} implies that $G$ has a subgroup  of index $2$ which satisfies the hypothesis of Theorem~\ref{Main} . Thus it suffices to prove the result for groups in which $|H|_2= 2^9$. This means that $S = QW$ or $S> QW$ and $\ov =S/Q \cong 3^{1+2}_+$. The latter situation is addressed in Lemma~\ref{index3} where is shown that if $S> QW$ then $G$ has a normal subgroup of index $3$ which also satisfies the hypothesis of Theorem~\ref{Main}. Thus we may assume that $S= QW$. Under this hypothesis in Section~10 we prove Theorem~\ref{NGECGr} which asserts that $C_G(r_1)= N_G(E)= KE\approx 2^{1+20}_+:\U_6(2)$.  Finally, in Section 11, we prove Theorem~\ref{G=M} which shows that under the hypothesis that $C_G(r_1)=N_G(E) = KE$, $G \cong {}^2\E_6(2)$. Thus we have $F^*(G) \cong {}^2\E_6(2)$ and the theorem is validated.

\end{document}